\documentclass{article}

\usepackage[centertags]{amsmath}
\usepackage{hyperref}
\usepackage{amsfonts}
\usepackage{amssymb}
\usepackage{amsthm}
\usepackage{newlfont}
\usepackage{amscd}
\usepackage{amsmath,amscd}
\usepackage{graphicx}
\usepackage[all]{xy}
\usepackage{verbatim}

\usepackage{color}

\newcommand{\barr}{\overline}

\newcommand{\NN}{\mathbb{N}}
\newcommand{\cA}{\mathcal{A}}
\newcommand{\cB}{\mathcal{B}}
\newcommand{\cC}{\mathcal{C}}
\newcommand{\cD}{\mathcal{D}}

\newcommand{\cF}{\mathcal{F}}

\newcommand{\cP}{\mathcal{P}}

\newcommand{\cW}{\mathcal{W}}

\newtheorem{thm}{Theorem}[section]

\newtheorem{cor}[thm]{Corollary}

\newtheorem{lem}[thm]{Lemma}
\newtheorem{prop}[thm]{Proposition}

\newtheorem{example}{Example}
\theoremstyle{definition}
\newtheorem{define}[thm]{Definition}
\theoremstyle{remark}
\newtheorem{rem}[thm]{Remark}

\DeclareMathOperator{\Pro}{Pro}

\DeclareMathOperator{\Hom}{Hom}

\DeclareMathOperator{\precolim}{colim}

\def\colim{\mathop{\precolim}}

\def \mcal{\mathcal}

\DeclareFontEncoding{OT2}{}{} 
\DeclareTextFontCommand{\textcyr}{\fontencoding{OT2}\fontfamily{wncyr}\fontseries{m}\fontshape{n}\selectfont}

\begin{document}
\title{A new model for pro-categories}

\author{
Ilan Barnea \footnote{The first author is supported by the Alexander von Humboldt Professorship of Michael Weiss of the University of Muenster.} \and
Tomer M. Schlank \footnote{The second author is supported by the Simons fellowship in the Department of Mathematics of the Massachusetts Institute of Technology.}}

\maketitle

\begin{abstract}
In this paper we present a new way to construct the pro-category of a category. This new model is very convenient to work with in certain situations. We present a few applications of this new model, the most important of which solves an open problem of Isaksen \cite{Isa} concerning the existence of functorial factorizations in what is known as the strict model structure on a pro-category. Additionally we explain and correct an error in one of the standard references on pro-categories.
\end{abstract}

\tableofcontents

\section{Introduction}

Pro-categories introduced by Grothendieck \cite{SGA4-I} have found many applications over the years in fields such as algebraic geometry \cite{AM}, shape theory \cite{MS} and more. Generally speaking, given a category $\cC$ one can think of $\Pro(\cC)$ as the category of ``inverse systems" in $\cC$. When $\cC$ has finite limits $\Pro(\cC)$ can be shown to be equivalent to the category of left exact functors from $\cC$ to the category of Sets.  While the last model has  some functorial advantages,  the model of $\Pro(\cC)$ as inverse systems is very concrete and pictorial. In this paper we suggest a new model for $\Pro(\cC)$ which we shall denote by $\barr{\Pro}(\cC)$. We think of $\barr{\Pro}(\cC)$ as a model in which one considers only inverse systems indexed by cofinite directed posets of infinite height. The big advantage of such indexing is that it is very susceptible  to proofs by induction. The category $\barr{\Pro}(\cC)$ is  the homotopy category of a very natural 2-category $\widetilde{\Pro}(\cC)$ which makes working with pro-categories quite  natural.

Specifically, we use the new model to prove a few propositions concerning factorizations of morphisms in pro-categories. These will later be used to deduce certain facts about model structures on pro-categories. The most important conclusion of this paper will be solving an open problem of Isaksen \cite{Isa} concerning the existence of functorial factorizations in what is known as the strict model structure on a pro-category. In order to state our results more accurately we give some definitions in a rather brief way. For a more detailed account see Section ~\ref{s:prelim}.

First recall that the category $\Pro(\mcal{C})$ has as objects all diagrams in $\cC$ of the form $I\to \cC$ such that $I$ is small and directed (see Definition ~\ref{d_directed}). The morphisms are defined by the formula:
$$\Hom_{\Pro(\mcal{C})}(X,Y):=\lim\limits_s \colim\limits_t \Hom_{\mcal{C}}(X_t,Y_s).$$
Composition of morphisms is defined in the obvious way.

Note that not every map in $\Pro(\cC)$ is a natural transformation (the source and target need not even have the same indexing category). However, every natural transformation between objects in $\Pro(\cC)$ having the same indexing category induces a morphism in $\Pro(\cC)$ between these objects, in a rather obvious way.

Let $M$ be a class of morphisms in $\cC$. We denote by $Lw^{\cong}(M)$ the class of morphisms in $\Pro(\mcal{C})$ that are \textbf{isomorphic} to a morphism that comes from a natural transformation which is a levelwise $M$-map.

If $T$ is a partially ordered set, then we view $T$ as a category which has a single morphism $u\to v$ iff $u\geq v$. A cofinite poset is a poset $T$ such that for every $x$ in $T$ the set $T_x:=\{z\in T| z \leq x\}$ is finite.

Suppose now that $\mcal{C}$ has finite limits. Let $T$ be a small cofinite poset and $F:X\to Y$ a morphism in $\mcal{C}^T$. Then $F$ will be called a special $M$-map, if the natural map $X_t \to Y_t \times_{\lim \limits_{s<t} Y_s} \lim \limits_{s<t} X_s $ is in $M$, for every $t$ in $T$. We denote by $Sp^{\cong}(M)$ the class of morphisms in $\Pro(\mcal{C})$ that are \textbf{isomorphic} to a morphism that comes from a (natural transformation which is a) special $M$-map.


We now define the 2-category $\widetilde{\Pro}(\cC)$. A (strict) 2-category is a category enriched in categories. More particularly  $\widetilde{\Pro}(\cC) $ is a category enriched in posets. Since a poset can be considered as a 1-category we indeed get a structure of a 2-category
on $\widetilde{\Pro}(\cC)$.
Let $A$ be a cofinite directed set. We will say that $A$ has infinite height if for every $a$ in $A$ there exists $a'$ in $A$ such that $a<a'$. An object of the 2-category $\widetilde{\Pro}(\cC)$ is a diagram $F:A\to \cC$, such that $A$ is a cofinite directed set of infinite height.
If $F:A\to \cC$ and $G:B\to\cC$ are objects in $\widetilde{\Pro}(\cC)$, a  1-morphism  $f$ from $F$ to $G$ is a defined to be a pair $f=(\alpha_f,\phi_f)$, such that $\alpha_f:B\to A$ is a strictly increasing function, and $\phi_f:\alpha_f^*F=F\circ\alpha_f\to G$ is a natural transformation.

Given two strictly increasing maps $\alpha,\alpha':B \to A$ we write $\alpha' \geq \alpha$ if  for every $b$ in $B$ we have $\alpha'(b)\geq\alpha(b)$.
Now we define the partial order on the set of 1-morphisms from $F$ to $G$. We set $(\alpha',\phi')\geq(\alpha,\phi)$ iff $\alpha' \geq \alpha$ and for  every $b$ in $B$ the following diagram commutes:
$$
\xymatrix{ & F(\alpha(b)) \ar[dr]^{\phi_b} & \\
F(\alpha'(b))\ar[ur]\ar[rr]^{\phi'_b} & & G(b)}
$$
(the arrow $F(\alpha'(b))\to F(\alpha(b))$ is of course the one induced by the unique morphism $\alpha'(b)\to \alpha(b)$ in $A$).

Composition of 1-morphisms in  $\widetilde{\Pro}(\cC)$ is defined by the formula:
$$(\beta,\psi)\circ(\alpha,\phi)=(\alpha\circ\beta,\psi\circ\phi_{\beta}).$$

We define $\barr{\Pro}(\cC)$ to be the homotopy category of $\widetilde{\Pro}(\cC)$, that is, the one obtained by identifying every couple of 1-morphisms with a 2-morphism between them. Namely, a morphism between $F$ and $G$  in $\barr{\Pro}(\cC)$ is a connected component of the poset $Mor_{\widetilde{\Pro}(\cC)}(F,G)$. We will show (see Corollary ~\ref{c_directed}) that every such connected component is a directed poset.
Given a 1-morphism $f = (\alpha_f,\phi_f)$ in $\widetilde{\Pro}(\cC)$ we denote by $[f]= [\alpha_f,\phi_f]$ the corresponding morphism in $\barr{\Pro}(\cC)$.

There is a  natural functor:
$$i:\barr{\Pro}(\cC)\to \Pro(\cC),$$
the object function of this functor being the obvious one.
We will construct a functor
$S:\Pro(\cC)\to \barr{\Pro}(\cC)$
and prove (see Definition ~\ref{d:S} and Corollary ~\ref{c_equiv}):
\begin{prop}\label{p:equ_0}
The pair of functors:
$$i:\barr{\Pro}(\cC)\rightleftarrows \Pro(\cC):S$$
are inverse equivalences  of categories.
\end{prop}

In the proof of Proposition ~\ref{p:equ_0} we will use the classical theorem saying that for every small directed category $I$ there exists a cofinite directed set $A$ and a cofinal functor: $p:A\to I.$ In \cite{Isa}, Isaksen gives two references to this theorem: one is \cite{EH} Theorem 2.1.6 and the other is \cite{SGA4-I} Proposition 8.1.6. We take this opportunity to explain and correct a slight error in the proof given in \cite{EH} (see the discussion following Corollary ~\ref{p_onto}).

Now let $M$ be a class of morphisms in $\cC$. It is easy to see that the pre-image of   $Lw^{\cong}(M)$ under $i$ is the class of morphisms in $\barr{\Pro}(\mcal{C})$ that are \textbf{isomorphic} to a morphism of the form $[id, \phi]$ where $\phi$ is levelwise in $M$.
Similarly, the pre-image of   $Sp^{\cong}(M)$ under $i$ is the class of morphisms in $\barr{\Pro}(\mcal{C})$ that are \textbf{isomorphic} to a morphism of the form $[id, \phi]$ where $\phi$ is a special $M$ map. In light of Proposition ~\ref{p:equ_0} we abuse notation and denote $Lw^{\cong}(M)=i^{-1}(Lw^{\cong}(M))$ and
$Sp^{\cong}(M)=i^{-1}(Sp^{\cong}(M))$.


Now let $\mcal{C}$ be a category and $M$ a class of morphisms in $\cC$. We denote by:
\begin{enumerate}
\item $R(M)$ the class of morphisms in $\mcal{C}$ that are retracts of morphisms in $M$.
\item $^{\perp}M$ the class of morphisms in $\cC$ having the left lifting property with respect to all maps in $M$.
\item $M^{\perp}$ the class of morphisms in $\cC$ having the right lifting property with respect to all maps in $M$.
\end{enumerate}

Let $N$ and $M$ be classes of morphisms in $\cC$. We will say that there exists a factorization in $\cC$ into a morphism in $N$ followed by a morphism in $M$ (and denote $Mor(\mcal{C}) = M\circ N$) if every map $X\to Y $ in $\mcal{C}$ can be factored as $X\xrightarrow{q} L\xrightarrow{p} Y $ such that $q$ is in ${N}$ and $p$ is in $M$. The pair $(N,M)$ will be called a weak factorization system in $\cC$ (see \cite{Rie}) if the following hold:
\begin{enumerate}
\item $Mor(\mcal{C}) = M\circ N$.
\item $N=^{\perp}M$.
\item $N^{\perp}=M$.
\end{enumerate}

A \emph{functorial} factorization in $\cC$ is a functor:
$\cC^{\Delta^1}\to\cC^{\Delta^2}$ denoted:
$$(X\xrightarrow{f}Y)\mapsto ({X}\xrightarrow{q_f} L_{f}\xrightarrow{p_f}Y),$$
such that:
\begin{enumerate}
\item For any morphism $f$ in $\cC$ we have: $f=p_f\circ q_f$.
\item For any morphism:
$$\xymatrix{{X}\ar[r]^{f}\ar[d]_{l} & {Y}\ar[d]^{k}\\
            {Z} \ar[r]^t   &   {W},}$$
in $\cC^{\Delta^1}$ the corresponding morphism in $\cC^{\Delta^2}$ is of the form:
$$\xymatrix{{X}\ar[r]^{q_f}\ar[d]_{l} & L_{f}\ar[r]^{h_f}\ar[d]^{L_{(l,k)}} & {Y}\ar[d]^{k}\\
            {Z} \ar[r]^{q_t} & L_t\ar[r]^{p_t}    &   {W}.}$$
\end{enumerate}

The above functorial factorization is said to be into a morphism in $N$ followed be a morphism in $M$ if for every morphism $f$ in $\cC$ we have that $q_f$ is in $N$ and $p_f$ is in $M$.

We will denote $Mor(\mcal{C}) = ^{func}M\circ N$ if there exists a \emph{functorial} factorization in $\cC$ into a morphism in $N$ followed by a morphism in $M$. The pair $(N,M)$ will be called a \emph{functorial} weak factorization system in $\cC$ if the following hold:
\begin{enumerate}
\item $Mor(\mcal{C}) =^{func} M\circ N$.
\item $N=^{\perp}M$.
\item $N^{\perp}=M$.
\end{enumerate}

Note that $Mor(\mcal{C}) =^{func} M\circ N$ clearly implies $Mor(\mcal{C}) = M\circ N$.

\begin{prop}\label{p:fact}
Let $\mcal{C}$ be a category that has finite limits, and let $N$ and $M$ be classes of morphisms in $\cC$. Then:
\begin{enumerate}
\item $^{\perp}R(Sp^{\cong}(M))=^{\perp}Sp^{\cong}(M)=^{\perp}M.$
\item If $Mor(\mcal{C}) = M\circ N$ then $Mor(\barr{\Pro}(\mcal{C})) = Sp^{\cong}(M)\circ Lw^{\cong}(N)$.
\item If $Mor(\mcal{C}) = M\circ N$ and $N\perp M$ (in particular, if $(N,M)$ is a weak factorization system in $\cC$), then $(Lw^{\cong}(N),R(Sp^{\cong}(M)))$ is a weak factorization system in $\barr{\Pro}(\cC)$.
\end{enumerate}
\end{prop}

The proof of Proposition ~\ref{p:fact} is strongly based on \cite{Isa} sections 4 and 5, and most of the ideas can be found there. The main novelty in this paper is the following theorem, proved in Section ~\ref{s_ff}:
\begin{thm}\label{t_main}
Let $\mcal{C}$ be a category that has finite limits, and let $N$ and $M$ be classes of morphisms in $\cC$. Then:
\begin{enumerate}
\item If $Mor(\mcal{C}) =^{func} M\circ N$ then $Mor(\barr{\Pro}(\mcal{C})) =^{func}  Sp^{\cong}(M)\circ Lw^{\cong}(N)$.
\item If $Mor(\mcal{C}) = ^{func}M\circ N$ and $N\perp M$ (in particular, if $(N,M)$ is a functorial weak factorization system in $\cC$), then $(Lw^{\cong}(N),R(Sp^{\cong}(M)))$ is a functorial weak factorization system in $\barr{\Pro}(\cC)$.
\end{enumerate}
\end{thm}

The factorizations constructed in the proof of Proposition ~\ref{p:fact} and Theorem ~\ref{t_main} both use Reedy type factorizations (see Section ~\ref{ss_reedy}). After passing from $\barr{\Pro}(C)$ to $\Pro(C)$, these are precisely the factorizations constructed by Edwards and Hastings in \cite{EH} and by Isaksen in \cite{Isa}. The main novelty here is that we show that these factorizations can be made functorial (given a functorial factorization in the original category). Here we use the convenience of working with    $\barr{\Pro}(\cC)$ as another model for $\Pro(\cC)$.

When working with pro-categories, it is frequently useful to have some kind of homotopy theory of pro-objects. Model categories, introduced in \cite{Qui}, provide a very general context in which it is possible to set up the basic machinery of homotopy theory. Given a category $\cC$, it is thus desirable to find conditions on $\cC$ under which $\Pro(\cC)$ can be given a model structure. It is natural to begin with assuming that $\cC$ itself has a model structure, and look for a model structure on $\Pro(\cC)$ which is in some sense induced by that of $\cC$. The following definition is based on the work of Edwards and Hastings \cite{EH}, Isaksen \cite{Isa} and others:

\begin{define}
Let $(\cC,\cW,\cF,\cC of)$ be a model category. The strict model structure on $\Pro(\cC)$ (if it exists) is defined by letting the acyclic cofibrations be $^{\perp}\cF$ and the cofibrations be  $^{\perp}(\cW\cap\cF)$.
\end{define}

This model structure is called the strict model structure on $\Pro(\cC)$ because several other model structures on the same category can be constructed from it through localization (which enlarges the class weak equivalences).

From Proposition ~\ref{p:fact} it clearly follows that in the strict model structure, if it exists, the cofibrations are given by $Lw^{\cong}(\cC of)$, the acyclic cofibrations are given by $Lw^{\cong}(\cW\cap\cC of)$, the fibrations are given by $R(Sp^{\cong}(\cF))$ and the acyclic fibrations are given by $R(Sp^{\cong}(\cF\cap\cW))$. The weak equivalences can then be characterized as maps that can be decomposed into an acyclic cofibration followed by an acyclic fibration.

Edwards and Hastings, in \cite{EH}, give sufficient conditions on a model category $\cC$ for the strict model structure on $\Pro(\cC)$ to exist. Isaksen, in \cite{Isa}, gives different sufficient conditions on $\cC$ and also shows that under these conditions the weak equivalences in the strict model structure on $\Pro(\cC)$ are given by $Lw^{\cong}(\cW)$.

\begin{rem}
It should be noted that we are currently unaware of any example of a model category $\cC$ for which one can show that the strict model structure on $\Pro(\cC)$ does not exist.
\end{rem}

The existence of the strict model structure implies that every map in $\Pro(\cC)$ can be factored into a (strict) cofibration followed by a (strict) trivial fibration, and into a (strict) trivial cofibration followed by a (strict) fibration. However, the existence of functorial factorizations of this form was not shown, and remained an open problem (see \cite{Isa}  Remark 4.10 and \cite{Cho}). The existence of functorial factorizations in a model structure is important for many constructions (such as framing, derived functor (between the model categories themselves) and more). In more modern treatments of model categories (such as \cite{Hov} or \cite{Hir}) it is even part of the axioms for a model structure.

From Theorem ~\ref{t_main} it clearly follows that if $\cC$ is a model category in the sense of \cite{Hov} or \cite{Hir}, that is, a model category with functorial factorizations, and if the strict model structure on $\Pro(\cC)$ exists, then the model structure on $\Pro(\cC)$ also admits functorial factorizations.

\subsection{Organization of the paper}
In Section ~\ref{s:prelim} we bring a short review of the necessary background on pro-categories. Some of the definitions and lemmas in this section are slightly non-standard.  In Section ~\ref{s:barr} we will define the category $\barr{\Pro}(\cC)$, and show its equivalence to $\Pro(\cC)$. We also explain and correct a slight error in \cite{EH} Theorem 2.1.6. In Section ~\ref{s_fact} we prove Proposition ~\ref{p:fact}. In Section ~\ref{s_ff} we prove our main result, namely Theorem ~\ref{t_main}.

\subsection{Acknowledgments}
We would like to thank the referee for his useful suggestions.

\section{Preliminaries on pro-categories}\label{s:prelim}

In this section we bring a short review of the necessary background on pro-categories. Some of the definitions and lemmas given here are slightly non-standard. For more details we refer the reader to \cite{AM}, \cite{EH}, and \cite{Isa}.

\begin{define}\label{d_directed}
A category $I$ is called \emph{directed} if the following conditions are satisfied:
\begin{enumerate}
\item The category $I$ is non-empty.
\item For every pair of objects $s$ and $t$ in $I$, there exists an object $u$ in $I$, together with
morphisms $u\to s$ and $u\to t$.
\item For every pair of morphisms $f,g:s\to t$ in $I$, there exists a morphism $h:u\to s$ in $I$, such that $f\circ h=g\circ h$.
\end{enumerate}
\end{define}

If $T$ is a partially ordered set, then we view $T$ as a category which has a single morphism $u\to v$ iff $u\geq v$. Note that this convention is opposite from the one used by some authors. Thus, a poset $T$ is directed iff $T$ is non-empty, and for every $a,b$ in $T$, there exists an element $c$ in $T$ such that $c\geq a,c\geq b$. In the following, instead of saying ``a directed poset" we will just say ``a directed set".

\begin{define}\label{def CDS}
A cofinite poset is a poset $T$ such that for every element $x$ in $T$ the set $T_x:=\{z\in T| z \leq x\}$ is finite.
\end{define}

\begin{define}\label{def_deg}
Let $A$ be a cofinite poset. We define the degree function of $A$: $d=d_A:A\to\NN$, by:
$$d(a):=max\{n\in \NN|\exists a_0<\cdots<a_n=a\}.$$
For every $n\geq -1$ we define: $A^n:=\{a\in A|d(a)\leq n\}$ $(A^{-1}=\phi)$.
\end{define}
Thus $d:A\to\NN$ is a strictly increasing function. The degree function enables us to define or prove things concerning $A$ inductively, since clearly: $A=\bigcup_{n\geq 0}A^n$. Many times in this paper, when defining (or proving) something inductively, we will skip the base stage. This is because we begin the induction from $n=-1$, and since $A^{-1}=\phi$ there is nothing to define (or prove) in this stage. The skeptic reader can check carefully the first inductive step to see that this is justified.

\begin{define}\label{d:section}
Let $T$ be a partially ordered set, and let $A$ be a subset of $T$. We shall say that $A$ is a (lower) \emph{section} of $T$, if for every $x$ in $A$ and $y$ in $T$ such that $y <x$, we have that $y$ is also in $A$.
\end{define}

\begin{example}
$T$ is a section of $T$. If $t$ is a maximal element in $T$, then $T \backslash \{t\}$ is a section of $T$. For any $t$ in $T$, the subset $T_t$ (see Definition ~\ref{def CDS}) is a section of $T$.
\end{example}

\begin{define}
Let $\cC$ be a category. The category $\cC^{\lhd}$ has as objects: $Ob(\cC)\coprod{\infty}$, and the morphisms are the morphisms in $\cC$, together with a unique morphism: $\infty\to c$, for every object $c$ in $\cC$.
\end{define}

In particular, if $\cC=\phi$ then $\cC^{\lhd}=\{\infty\}$.

Note that if $A$ is a cofinite poset and $a$ is an element in $A$ of degree $n$, then $A_a$ is naturally isomorphic to $(A_a^{n-1})^{\lhd}$ (where $A_a^{n-1}$ is just $(A_a)^{n-1}$, see Definition ~\ref{def CDS}).

The following lemma is clear, but we include it for later
reference.
\begin{lem}\label{l:eqiv_directed}
A cofinite poset $A$ is directed iff for every finite section $R$ of $A$ (see Definition ~\ref{d:section}),
there exists an element $c$ in $A$ such that $c\geq r$, for every $r$ in $R$.
A category $\cC$ is directed iff for every finite poset $R$, and for every functor $F:R\to \cC$, there exists an object $c$ in $\cC$, together with compatible morphisms $c\to F(r)$, for every $r$ in $R$ (that is, a morphism $Diag(c)\to F$ in $\cC^R$, or equivalently we can extend the functor $F:R\to \cC$ to a functor $R^{\lhd}\to \cC$).
\end{lem}

A category is called small if it has a small set of objects and a small set of morphisms

\begin{define}\label{def_pro}
Let $\mcal{C}$ be a category. The category $\Pro(\mcal{C})$ has as objects all diagrams in $\cC$ of the form $I\to \cC$ such that $I$ is small and directed (see Definition ~\ref{d_directed}). The morphisms are defined by the formula:
$$\Hom_{\Pro(\mcal{C})}(X,Y):=\lim\limits_s \colim\limits_t \Hom_{\mcal{C}}(X_t,Y_s).$$
Composition of morphisms is defined in the obvious way.
\end{define}

Thus, if $X:I\to \mcal{C}$ and $Y:J\to \mcal{C}$ are objects in $\Pro(\mcal{C})$, giving a morphism $X\to Y$ means specifying, for every $s$ in $J$, a morphism $X_t\to Y_s$ in $\mcal{C}$, for some $t$ in $I$. These morphisms should of course satisfy some compatibility condition. In particular, if the indexing categories are equal: $I=J$, then any natural transformation: $X\to Y$ gives rise to a morphism $X\to Y$ in $\Pro(C)$. More generally, if $p:J\to I$ is a functor, and $\phi:p^*X:=X\circ p\to Y$ is a natural transformation, then the pair $(p,\phi)$ determines a morphism $\nu_{p,\phi}:X\to Y$ in $\Pro(C)$ (for every $s$ in $J$ we take the morphism $\phi_s:X_{p(s)}\to Y_s$). In particular, taking $Y=p^*X$ and $\phi$ to be the identity natural transformation, we see that $p$ determines a morphism $\nu_{p,X}:X\to p^*X$ in $\Pro(C)$.

Let $f:X\to Y$ be a morphism in $\Pro(\cC)$. A morphism in $\cC$ of the form $X_r\to Y_s$, that represents the $s$ coordinate of $f$ in $\colim\limits_{t\in I} \Hom_{\mcal{C}}(X_t,Y_s)$, will be called ``representing $f$".

The word pro-object refers to objects of pro-categories. A \emph{simple} pro-object
is one indexed by the category with one object and one (identity) map. Note that for any category $\mcal{C}$, $\Pro(\mcal{C})$ contains $\mcal{C}$ as the full subcategory spanned by the simple objects.

\begin{define}\label{d:cofinal}
Let $p:J\to I$ be a functor between small categories. The functor $p$ is said to be (left) \emph{cofinal} if for every $i$ in $I$, the over category ${p}_{/i}$ is nonempty and connected.
\end{define}

Cofinal functors play an important role in the theory of pro-categories mainly because of the following well known lemma (see for example \cite{EH}):

\begin{lem}\label{l:cofinal}
Let $p:J\to I$ be a cofinal functor between small directed categories, and let $X:I\to \cC$ be an object in $\Pro(\cC)$. Then the morphism in $\Pro(\cC)$ defined by $p$: $\nu_{p,X}:X\to p^*X$ is an isomorphism.
\end{lem}

\begin{define}\label{def natural}
Let $\mcal{C}$ be a category with finite limits, $M$ a class of morphisms in $\mcal{C}$, $I$ a small category and $F:X\to Y$ a morphism in $\mcal{C}^I$. Then $F$ will be called:
\begin{enumerate}
\item A \emph{levelwise} $M$-\emph{map}, if for every $i$ in $I$: the morphism $X_i\to Y_i$ is in $M$. We will denote this by $F\in Lw(M)$.
\item A \emph{special} $M$-\emph{map}, if the following hold:
    \begin{enumerate}
    \item The indexing category $I$ is a cofinite poset (see Definition ~\ref{def CDS}).
    \item The natural map $X_t \to Y_t \times_{\lim \limits_{s<t} Y_s} \lim \limits_{s<t} X_s $ is in $M$, for every $t$ in $ I$.
    \end{enumerate}
    We will denote this by $F\in Sp(M)$.
\end{enumerate}
\end{define}

Let $\mcal{C}$ be a category. Given two morphisms $f,g$ in $\cC$ we denote by $f \perp g$ the fact
that $f$ has the left lifting property with respect to $g$.
If $M,N$ are classes of morphisms in $\cC$, we denote by $M \perp N$ the fact that $f\perp g$ for every $f $ in $M$ and $g$ in $N$.

\begin{define}\label{def mor}
Let $\mcal{C}$ be a category with finite limits, and $M \subseteq Mor(\mcal{C})$ a class of morphisms in $\mcal{C}$. Denote by:
\begin{enumerate}
\item $R(M)$ the class of morphisms in $\mcal{C}$ that are retracts of morphisms in $M$. Note that $R(R(M))=R(M)$.
\item ${}^{\perp}M$ the class of morphisms in $\mcal{C}$ having the left lifting property with respect to all the morphisms in $M$.
\item $M^{\perp}$ the class of morphisms in $\mcal{C}$ having the right lifting property with respect to all the morphisms in $M$.
\item $Lw^{\cong}(M)$ the class of morphisms in $\Pro(\mcal{C})$ that are \textbf{isomorphic} to a morphism that comes from a natural transformation which is a levelwise $M$-map.
\item $Sp^{\cong}(M)$ the class of morphisms in $\Pro(\mcal{C})$ that are \textbf{isomorphic} to a morphism that comes from a natural transformation which is a special $M$-map.
\end{enumerate}
\end{define}

Note that:
$$
(M \subset  {}^\perp N) \Leftrightarrow (N \subset   M^\perp) \Leftrightarrow (M \perp N).
$$

The following lemma appears in \cite{Isa}, Proposition 2.2. We include it here for later reference.
\begin{lem}\label{l:ret_lw}
Let $M$ be any class of morphisms in $\mcal{C}$. Then $$R(Lw^{\cong}(M)) = Lw^{\cong}(M).$$
\end{lem}

The following lemma is an easy diagram chase. We include it for later reference.
\begin{lem}\label{c:ret_lift}
Let $M$ be any class of morphisms in $\mcal{C}$. Then:
$$(R(M))^{\perp} = M^{\perp},\; {}^{\perp}(R(M)) = {}^{\perp}M,$$
$$R(M^{\perp}) = M^{\perp},\; R({}^{\perp}M) = {}^{\perp}M.$$
\end{lem}

\begin{lem}\label{l:SpMo_is_Mo}
${}^{\perp}Sp^{\cong}(M) = {}^{\perp}M$.
\end{lem}

\begin{rem}
The idea of the proof of Lemma ~\ref{l:SpMo_is_Mo} appears in \cite{Isa} (see the proof of Lemma 4.11).
\end{rem}

\begin{proof}
Since $M\subseteq Sp^{\cong}(M)$, it is clear that ${}^{\perp}Sp^{\cong}(M)\subseteq {}^{\perp}M$. It remains to show that ${}^{\perp}Sp^{\cong}(M)\supseteq {}^{\perp}M$.
Let $g $ be in ${}^{\perp}M$ and $f $ in $Sp^{\cong}(M)$. We need to show that $g \perp f$. Without loss of generality we may assume that $f$ comes from a natural transformation $X\to Y$ with the following properties:
    \begin{enumerate}
    \item The indexing category is a cofinite directed set: $T$.
    \item The natural map $X_t \to Y_t \times_{\lim \limits_{s<t} Y_s} \lim \limits_{s<t} X_s $ is in $M$ for every $t$ in $T$.
    \end{enumerate}
We need to construct a lift in the following diagram:
$$
\xymatrix{
A \ar[d]^g \ar[r] & \{X_t\} \ar[d]^f \\
B          \ar[r] & \{Y_t\}.
}$$
Giving a morphism $B\to \{X_t\}$ means giving morphisms $B\to X_t$ for every $t$ in $T$, compatible relative to morphisms in $T$, where $X_t$ is regarded as a simple object in $\Pro(\mcal{C})$. Thus, it is enough to construct compatible lifts $B\to X_t$, in the diagrams:
$$
\xymatrix{
A \ar[d]^g \ar[r] & X_t \ar[d]^{f_t} \\
B          \ar[r] & Y_t
}$$
for every $t$ in $T$.

We will do this by induction on $t$. If $t$ is an element of $T$ such that $d(t)= 0$ (that is, $t$ is a minimal element of $T$), then such a lift exists since $g $ is in ${}^{\perp}M$, and $$X_t \to  Y_t \times_{\lim \limits_{s<t}  Y_s} \lim \limits_{s<t}  X_s = Y_t$$ is in $M$. Suppose that we have constructed compatible lifts $B\to X_s$, for every $s<t$. Let us construct a compatible lift $B \to X_t$.

We will do this in two stages. First, the compatible lifts $B\to X_s$, for $s<t$, available by the induction hypothesis, gather together to form a lift:
$$
\xymatrix{
A \ar[d]^g \ar[r] & \lim \limits_{s<t} X_s \ar[d]^{f} \\
B  \ar[ru] \ar[r] & \lim \limits_{s<t} Y_s
}$$
and the diagram
$$
\xymatrix{
B\ar[d]  \ar[r]& Y_t \ar[d]\\
\lim \limits_{s<t} X_s \ar[r]   & \lim \limits_{s<t} Y_s
}$$
obviously commutes (since the morphisms $B\to Y_t$ are compatible). Thus we get a lift
$$
\xymatrix{
A\ar[d]^g  \ar[r]&        Y_t \times_{\lim \limits_{s<t} Y_s} \lim \limits_{s<t} X_s \ar[d]\\
B \ar[r] \ar[ru]    & Y_t.
}$$
The second stage is to choose any lift in the square:
$$
\xymatrix{
A\ar[d]^g\ar[r]  \ar[r] & X_t \ar[d]  \\
B \ar[r] & Y_t \times_{\lim \limits_{s<t} Y_s} \lim \limits_{s<t} X_s
}$$
which exists since $g$ is in ${}^\perp M$, and $X_t \to Y_t \times_{\lim \limits_{s<t} Y_s} \lim \limits_{s<t} X_s$ is in $M$.
In particular, we get that the following diagram commutes:
$$
\xymatrix{
B \ar[dr] \ar[r]& X_t \ar[d]\\
 & \lim \limits_{s<t} X_s,
}$$
which shows that the lift $B\to X_t$ is compatible.
\end{proof}

\subsection{Constructing inverse equivalences}
In this subsection we present a construction that produces an inverse equivalence to a fully faithful functor, given some extra data. We will use this construction a couple of times in this paper.

\begin{define}\label{d:inverse}
Let $F:\cC\to \cD$ be a fully faithful functor between categories. Suppose we are given two class functions: $g:Ob(\cD)\to Ob(\cC)$ and $\phi:Ob(\cD)\to Mor(\cD)$, such that for every object $d$ in $\cD$ we have that:
$$\phi(d):d\xrightarrow{\cong}F(g(d))$$
is an isomorphism.

We define a functor $G=G_{F,g,\phi}:\cD\to \cC$, as follows:

For every object $d$ in $\cD$ we define $G(d) := g(d)\in Ob(\cC).$

Let $f:d\to d'$ be a morphism in $\cD$. Since $F$ is fully faithful, the function:
$$F_{(G(d),G(d'))}:Hom_{\cC}(G(d),G(d'))\to Hom_{\cD}(F(G(d)),F(G(d')))$$
is bijective. Thus we have an inverse function:
$$F_{(G(d),G(d'))}^{-1}:Hom_{\cD}(F(G(d)),F(G(d')))\to Hom_{\cC}(G(d),G(d')).$$
We note that since $F_{(G(d),g(d'))}$ is bijective, the inverse function is well defined and can be constructed without using the axiom of choice.

We now define:
$$G(f):=F_{(G(d),G(d'))}^{-1}\left(\phi(d')\circ f\circ\phi(d)^{-1}\right):G(d)\to G(d').$$

It is not hard to verify that $G:\cD\to \cC$ is indeed a functor.
\end{define}

The following lemma is a straightforward verification:
\begin{lem}\label{l:inverse}
Let $F:\cC\to \cD$ be a fully faithful functor, and let $g:Ob(\cD)\to Ob(\cC)$ and $\phi:Ob(\cD)\to Mor(\cD)$ be two class functions as in Definition ~\ref{d:inverse}. Then the functor $G_{F,g,\phi}:\cD\to \cC$ constructed in Definition ~\ref{d:inverse} is an inverse equivalence to $F$ (that is, the compositions of $F$ and $G$ are naturally isomorphic to the identity functors).
\end{lem}

\begin{rem}
Given a fully faithful functor $F:\cC\to \cD$ and  two class functions $g:Ob(\cD)\to Ob(\cC)$ and $\phi:Ob(\cD)\to Mor(\cD)$ as in Definition ~\ref{d:inverse}, it is clear that $F$ is essentially surjective on objects.
Thus, by a classical theorem in category theory (see for example \cite{ML}), there exists a functor $G:\cD\to\cC$ that is an inverse equivalence to $F$. The purpose of Definition ~\ref{d:inverse} and Lemma ~\ref{l:inverse} is to give an explicit construction of such an inverse, and to emphasize the constructive nature of this construction.

In other words, given a fully faithful essentially surjective functor $F:\cC\to \cD$, an application of the axiom of choice for classes can produce $g:Ob(\cD)\to Ob(\cC)$ and $\phi:Ob(\cD)\to Mor(\cD)$ as in Definition ~\ref{d:inverse}. However, once we are given the class functions $g$ and $\phi$, we can always construct the inverse equivalence $G:\cD\to \cC$ \emph{without using the axiom of choice}. Thus, if we are able to construct the class functions $g$ and $\phi$ without using the axiom of choice (as is the case in our applications here), then we can also construct the inverse equivalence $G$ constructively.
\end{rem}


\section{A new model for a pro-category}\label{s:barr}

In this section we will define a category $\barr{\Pro}(\cC)$, and show that this category is equivalent to $\Pro(\cC)$. This category can be thought of as a new model for the pro-category of $\cC$.

This model seems to have some advantages over the traditional models for a pro-category. As an application of this new
model, we will use it to construct functorial factorizations in pro-categories in Section ~\ref{s_ff}.

Throughout this section we let $\cC$ be an arbitrary category.

\subsection{Definition of $\widetilde{\Pro}(\cC)$ and $\barr{\Pro}(\cC)$}
The purpose of this subsection is to define the 2-category $\widetilde{\Pro}(\cC)$. A 2-category in a category enriched in categories. More particularly, $\widetilde{\Pro}(\cC) $ is a category enriched in posets. Since a poset can considered as a 1-category we indeed get a structure of a 2-category
on $\widetilde{\Pro}(\cC)$.

\begin{define}
Let $A$ be a poset. We will say that $A$ has \emph{infinite height} if for every $a$ in $A$ there exists $a'$ in $A$ such that $a<a'$.
\end{define}

An object of the 2-category $\widetilde{\Pro}(\cC)$ is a diagram $F:A\to \cC$, such that $A$ is a cofinite directed set of infinite height. If we say that $F^A$ is an object in $\widetilde{\Pro}(\cC)$, we will mean that $F$ is an object of $\widetilde{\Pro}(\cC)$ and $A$ is its domain.
If $F^A$ and $G^B$ are objects in $\widetilde{\Pro}(\cC)$, a 1-morphism  $f$ from $F$ to $G$ is defined to be a pair $f=(\alpha_f,\phi_f)$, such that $\alpha_f:B\to A$ is a strictly increasing function, and $\phi_f:\alpha_f^*F=F\circ\alpha_f\to G$ is a natural transformation.

\begin{rem}\label{r_strictly}
The reason for demanding a \emph{strictly} increasing function in the definition of a 1-morphism will not be clear until much later. See for example the construction of the functor: $\barr{\Pro}(\cC^{\Delta^1})\to \barr{\Pro}(\cC^{\Delta^2})$ in Section ~\ref{s_ff}.
\end{rem}

Given two strictly increasing maps $\alpha,\alpha':B \to A$ we write $\alpha' \geq \alpha$ if  for every $b$ in $B$ we have $\alpha'(b)\geq\alpha(b)$.
Now we define a partial order on the set of 1-morphisms from $F$ to $G$. We set $(\alpha',\phi')\geq(\alpha,\phi)$ iff $\alpha' \geq \alpha$ and for  every $b$ in $B$ the following diagram commutes:
$$
\xymatrix{ & F(\alpha(b)) \ar[dr]^{\phi_b} & \\
F(\alpha'(b))\ar[ur]\ar[rr]^{\phi'_b} & & G(b)}
$$
(the arrow $F(\alpha'(b))\to F(\alpha(b))$ is of course the one induced by the unique morphism $\alpha'(b)\to \alpha(b)$ in $A$).

Composition of 1-morphisms in  $\widetilde{\Pro}(\cC)$ is defined by the formula:
$$(\beta,\psi)\circ(\alpha,\phi)=(\alpha\circ\beta,\psi\circ\phi_{\beta}).$$
It is not hard to check that we have turned the set of 1-morphisms from $F$ to $G$ into a poset and that
$\widetilde{\Pro}(\cC)$ is enriched in posets.

\begin{rem}
Using the language of 2-categories one can define $\widetilde{\Pro}(\cC)$ as a certain 2-comma category. To state the claim accurately it will be more convenient to consider the dual case of ind-categories. Everything we did in this paper is completely dualizable, so one can define the 2-category $\widetilde{Ind}(\cC)$ in an obvious way. However, when working with ind-categories it is more convenient to view a poset $T$ as a category which has a single morphism $u\to v$ iff $u\leq v$.

Let $\mathbb{P}$ denote the category enriched in posets with
$Ob (\mathbb{P})$ being all cofinite directed posets of infinite height and
$\Hom_{\mathbb{P}}(A,B)$ the poset of strictly increasing maps $A \to B$. The 2-category $\mathbb{P}$ is a sub 2-category of $Cat$ so there is a natural strict 2-functor $\mathbb{P}\hookrightarrow Cat$. There is also a strict 2-functor $\{\cC\}\hookrightarrow Cat$ choosing the category $\cC$.
Then $\widetilde{Ind}(\cC)$ can be described as the 2-comma category (see \cite{Gra} p. 29) of the above pair of 2-functors:
$$\widetilde{Ind}(\cC)\simeq\mathbb{P}\downarrow^2\cC.$$
\end{rem}

We define $\barr{\Pro}(\cC)$ to be the homotopy category of $\widetilde{\Pro}(\cC)$. That is, the one obtained by identifying every couple of 1-morphisms with a 2-morphism between them. Namely, a morphism between $F$ and $G$  in $\barr{\Pro}(\cC)$ is a connected component of the poset $Mor_{\widetilde{\Pro}(\cC)}(F,G)$. We will show (see Corollary ~\ref{c_directed}) that every such connected component is a directed poset.
Given a 1-morphism $f = (\alpha_f,\phi_f)$ in $\widetilde{\Pro}(\cC)$ we denote by $[f]= [\alpha_f,\phi_f]$ the corresponding morphism in $\barr{\Pro}(\cC)$.

In particular, if $F,G$ are objects in $\barr{\Pro}(C)$ having equal indexing categories, then any natural transformation: $\phi:F^A\to G^A$ gives rise to a morphism $[id, \phi]:F^A\to G^A$ in $\barr{\Pro}(C)$. If $F^A$ is any object in $\barr{\Pro}(C)$ and $\alpha:B\to A$ is a strictly increasing map between cofinite directed sets of infinite height, then $\alpha$ determines a morphism $[\alpha, id]:F^A\to \alpha^*F^A$ in $\barr{\Pro}(C)$.



\subsection{Equivalence of $\barr{\Pro}(\cC)$ and $\Pro(\cC)$}

In this subsection we construct a natural functor $i:\barr{\Pro}(\cC)\to \Pro(\cC)$. We then show that $i$ is a categorical equivalence.

Let $F:A\to \cC$ be an object in $\barr{\Pro}(\cC)$. Then clearly $i(F):=F$ is also an object ${\Pro}(\cC)$.

Let $F^A$ and $G^B$ be objects in $\barr{\Pro}(\cC)$, and let $(\alpha,\phi)$ be a 1-morphism from $F$ to $G$. Then $(\alpha,\phi)$ determines a morphism $F\to G$ in $\Pro(\cC)$ (for every $b$ in $B$ take the morphism $\phi_b:F_{\alpha(b)}\to G_b$). Suppose now that $(\alpha',\phi')$ is another 1-morphism from $F$ to $G$, such that $(\alpha',\phi')\geq(\alpha,\phi)$. Then it is clear from the definition of the partial order on 1-morphisms that for every $b$ in $ B$ the morphisms $\phi_b:F(\alpha(b))\to G(b)$ and $\phi'_b:F(\alpha'(b))\to G(b)$ represent the same object in $\colim_{i\in A}\Hom_{\cC}(F(i),G(b))$. Thus $(\alpha',\phi')$ and $(\alpha,\phi)$ determine the same morphism $F\to G$ in $\Pro(\cC)$. It follows, that a morphism $F\to G$ in $\barr{\Pro}(\cC)$ determines a well defined morphism $i(F)\to i(G)$ in $\Pro(\cC)$ through the above construction. This construction clearly commutes with compositions and identities, so we have defined a functor: $i:\barr{\Pro}(\cC)\to \Pro(\cC)$.

\begin{prop}\label{p_full}
The functor $i:\barr{\Pro}(\cC)\to \Pro(\cC)$ is full.
\end{prop}
\begin{proof}
Let $F^A$ and $G^B$ be objects in $\barr{\Pro}(\cC)$. Let $f:F\to G$ be a morphism in $\Pro(\cC)$. We need to construct a 1-morphism $(\alpha,\phi)$ from $F$ to $G$ that induces our given $f$.

We will define $\alpha:B\to A$, and $\phi:F\circ\alpha\to G$ recursively.

Let $n\geq 0$. Suppose we have defined a strictly increasing function $\alpha:B^{n-1}\to A$, and a natural transformation $\phi:F\circ\alpha\to G|_{B^{n-1}}$, such that for every $b$ in $B^{n-1}$ the morphism $\phi_b:F(\alpha(b))\to G(b)$ represents $f$ (see Definition ~\ref{def_deg} and the remarks after Definition ~\ref{def_pro}).

Let $b$ be an element in $B^n \setminus B^{n-1}$.
Write $B_b^{n-1}=\{b_1,...,b_k\}$. We will prove the following by induction on $i$:

For every $i=0,...,k$ there exists $a_i$ in $A$ and a morphism $F(a_i)\to G(b)$ representing $f$, such that for every $1\leq j\leq i$ we have $a_i\geq \alpha(b_j)$ and the following diagram commutes:
$$\xymatrix{F(a_i)\ar[r]\ar[d]_{F(a_{i}\to \alpha(b_{j}))} & G(b)\ar[d]^{G(b\to b_{j})} \\
F(\alpha(b_j))\ar[r]^{\phi_{b_j}} & G(b_j).}$$

$i=0$. Choose $a_0$ in $A$ and a morphism $F(a_0)\to G(b)$ representing $f$.

Suppose we have proved the above for some $i\in\{0,...,k-1\}$.

We will prove the above for $i+1$. The morphisms $F(a_i)\to G(b)$ and $\phi_{b_{i+1}}:F(\alpha(b_{i+1}))\to G(b_{i+1})$ both represent $f$. We have that $b\geq b_{i+1}$, so the compatibility of the representing morphisms implies that $\phi_{b_{i+1}}$ and the composition
$$F(a_i)\to G(b)\xrightarrow{G(b\to b_{i+1})}G(b_{i+1})$$
represents the same element in $\colim_{a\in A}\Hom_{\cC}(F(a),G(b_{i+1}))$. Thus, there exists $a_{i+1}$ in $A$ such that $a_{i+1}\geq a_i, \alpha(b_{i+1})$ and the following diagram commutes:
$$\xymatrix{                                                                & F(a_i)\ar[dr]            &                               & \\
F(a_{i+1})\ar[ur]^{F(a_{i+1}\to a_i)}\ar[dr]_{F(a_{i+1}\to \alpha(b_{i+1}))}&                          & G(b)\ar[dr]^{G(b\to b_{i+1})} & \\
                                                                            & F(\alpha(b_{i+1}))\ar[rr]^{\phi_{b_{i+1}}} &             & G(b_{i+1}).}$$
It is not hard to verify that taking $F(a_{i+1})\to G(b)$ to be the morphism described in the diagram above finishes the inductive step.

Since $A$ has infinite height we can find $\alpha(b)$ in $A$ such that $\alpha(b)>a_k$. Defining $\phi_b$ to be the composition:
$$F(\alpha(b))\xrightarrow{F(\alpha(b)\to a_{k})}F(a_{k})\to G(b)$$
finishes the inductive step.

We now define $\alpha(b):=a_k$. Thus $\alpha(b)$ is an object in $A$ and there exists a morphism $\phi_b:F(\alpha(b))\to G(b)$ representing $f$, such that for every $b'$ in $B_b^{n-1}$ (see Definition ~\ref{def CDS}) we have $\alpha(b)> \alpha(b')$ and the following diagram commutes:
$$\xymatrix{F(\alpha(b))\ar[r]^{\phi_{b}}\ar[d]_{F(\alpha(b)\to \alpha(b'))} & G(b)\ar[d]^{G(b\to b_{i+1})} \\
F(\alpha(b'))\ar[r]^{\phi_{b'}} & G(b').}$$
This completes the recursive definition, and thus the proof of the proposition.
\end{proof}

We now wish to prove that $i$ is faithful. We will prove a stronger result:
\begin{prop}\label{p_directed}
Let $F^A$ and $G^B$ be objects in $\barr{\Pro}(\cC)$, and let $(\alpha,\phi)$ and $(\alpha',\phi')$ be 1-morphisms from $F$ to $G$. Assume that $(\alpha,\phi)$ and $(\alpha',\phi')$ induce the same morphism $f:F\to G$ in $\Pro(\cC)$. Then there exists a  1-morphism $(\alpha'',\phi'')$ from $F$ to $G$ such that $(\alpha'',\phi'')\geq(\alpha,\phi),(\alpha',\phi')$.
\end{prop}
\begin{proof}
We will define $\alpha'':B\to A$ and $\phi'':F\circ\alpha\to G$ recursively.

Let $n\geq 0$. Suppose we have defined a strictly increasing function $\alpha'':B^{n-1}\to A$ and a natural transformation $\phi'':F\circ\alpha''\to G|_{B^{n-1}}$, such that for every $b$ in $ B^{n-1}$ we have $\alpha''(b)\geq \alpha(b),\alpha'(b)$ and the following diagram commutes:
$$\xymatrix{ &   F(\alpha(b))\ar[dr]^{\phi_b}     & \\
F(\alpha''(b))\ar[ur]^{F(\alpha''(b)\to \alpha(b))}\ar[dr]_{F(\alpha''(b)\to \alpha'(b))} \ar[rr]^{\phi''_b} & & G(b) \\
             &   F(\alpha'(b))\ar[ur]_{\phi'_b}   & }$$
(see Definition ~\ref{def_deg}).

Let $b$ be an element in $B^n \setminus B^{n-1}$.
Write $B_b^{n-1}=\{b_1,...,b_k\}$. We will prove the following by induction on $i$:

For every $i=0,...,k$ there exists $a_i$ in $A$ and a morphism $F(a_i)\to G(b)$, such that for every $1\leq j\leq i$ we have $a_i\geq \alpha''(b_j)$ and the following diagram commutes:
$$\xymatrix{F(a_i)\ar[r]\ar[d]_{F(a_{i}\to \alpha''(b_{j}))} & G(b)\ar[d]^{G(b\to b_{j})} \\
F(\alpha''(b_j))\ar[r]^{\phi''_{b_j}} & G(b_j),}$$
and we have $a_i \geq \alpha(b), \alpha'(b)$ and the following diagram commutes:
$$\xymatrix{ &   F(\alpha(b))\ar[dr]^{\phi_b}     & \\
F(a_i)\ar[ur]^{F(a_i\to \alpha(b))}\ar[dr]_{F(a_i\to \alpha'(b))}\ar[rr] & & G(b). \\
             &   F(\alpha'(b))\ar[ur]_{\phi'_b}   & }$$

$i=0$. The morphisms $\phi_b:F(\alpha(b))\to G(b)$ and $\phi'_b:F(\alpha'(b))\to G(b)$ represent the same element in $\colim_{a\in A}\Hom_{\cC}(F(a),G(b))$. It follows that there exists $a_0$ in $A$ such that $a_0\geq \alpha(b), \alpha'(b)$ and the following diagram commutes:
$$\xymatrix{ &   F(\alpha(b))\ar[dr]^{\phi_b}     & \\
F(a_0)\ar[ur]^{F(a_0\to \alpha(b))}\ar[dr]_{F(a_0\to \alpha'(b))} & & G(b). \\
             &   F(\alpha'(b))\ar[ur]_{\phi'_b}   & }$$
We thus \emph{define} the morphism $F(a_0)\to G(b)$ to be the one described in the diagram above.

Suppose we have proved the above for some $i\in\{0,...,k-1\}$.

We will prove the above for $i+1$. The morphisms $F(a_i)\to G(b)$ and $\phi''_{b_{i+1}}:F(\alpha''(b_{i+1}))\to G(b_{i+1})$ both represent $f$. We have that $b\geq b_{i+1}$, so the compatibility of the representing morphisms implies that $\phi''_{b_{i+1}}$ and the composition
$$F(a_i)\to G(b)\xrightarrow{G(b\to b_{i+1})}G(b_{i+1})$$
represent the same object in $\colim_{a\in A}\Hom_{\cC}(F(a),G(b_{i+1}))$. Thus, there exists $a_{i+1}$ in $A$ such that $a_{i+1}\geq a_i, \alpha''(b_{i+1})$ and the following diagram commutes:
$$\xymatrix{                                                                & F(a_i)\ar[dr]            &                               & \\
F(a_{i+1})\ar[ur]^{F(a_{i+1}\to a_i)}\ar[dr]_{F(a_{i+1}\to \alpha''(b_{i+1}))}&                          & G(b)\ar[dr]^{G(b\to b_{i+1})} & \\
                                                                            & F(\alpha''(b_{i+1}))\ar[rr]^{\phi''_{b_{i+1}}} &             & G(b_{i+1})}.$$
It is not hard to verify that taking $F(a_{i+1})\to G(b)$ to be the morphism described in the diagram above finishes the inductive step.

Since $A$ has infinite height we can find $\alpha''(b)$ in $A$ such that $\alpha''(b)>a_k$. Defining $\phi''_b$ to be the composition:
$$F(\alpha''(b))\xrightarrow{F(\alpha''(b)\to a_{k})}F(a_{k})\to G(b)$$
finishes the inductive step.

We now define $\alpha''(b):=a_k$. Thus $\alpha''(b)$ is an object in $A$ and
there exists a morphism $\phi''_b:F(\alpha''(b))\to G(b)$, such that for every $b'$ in $ B^{n-1}_b$ (see Definition ~\ref{def CDS}) we have $\alpha''(b)> \alpha''(b')$ and the following diagram commutes:
$$\xymatrix{F(\alpha''(b))\ar[r]^{\phi''_b}\ar[d]_{F(\alpha''(b)\to \alpha''(b'))} & G(b)\ar[d]^{G(b\to b')} \\
F(\alpha''(b'))\ar[r]^{\phi''_{b'}} & G(b'),}$$
and we have $\alpha''(b)\geq \alpha(b), \alpha'(b)$ and the following diagram commutes:
$$\xymatrix{ &   F(\alpha(b))\ar[dr]^{\phi_b}     & \\
F(\alpha''(b))\ar[ur]^{F(\alpha''(b)\to \alpha(b))}\ar[dr]_{F(\alpha''(b)\to \alpha'(b))}\ar[rr]^{\phi''_b} & & G(b). \\
             &   F(\alpha'(b))\ar[ur]_{\phi'_b}   & }$$
This completes the recursive definition, and thus the proof of the proposition.
\end{proof}

\begin{cor}\label{c_faithful}
The functor $i:\barr{\Pro}(\cC)\to \Pro(\cC)$ is faithful.
\end{cor}

\begin{cor}\label{c_directed}
Let $F$ and $G$ be objects in $\barr{\Pro}(\cC)$. Then every connected component of the poset of 1-morphisms from $F$ to $G$ (that is, every morphism from $F$ to $G$ in $\barr{\Pro}(\cC)$) is directed.
\end{cor}


\begin{define}\label{d_onto}
Let $I$ be a small directed category. We will describe a construction that produces a small cofinite poset $A_I$ and a functor: $p_I:A_I\to I$.

We shall define $A_I$ and $p_I:A_I\to I$ recursively.

We start with defining $A_I^{-1}:=\emptyset$, and $p_I^{-1}:A_I^{-1}=\emptyset \to I$ in the only possible way.

Now, suppose we have defined an $n$-level cofinite poset $A_I^n$, and a functor $p_I^n:A_I^n \to I$.

We define $B_I^{n+1}$ to be the set of all tuples  $(R,p:R^{\lhd}\to I)$ such that
$R$ is a finite section in $A_I^n$ (see Definition ~\ref{d:section}), $p:R^{\lhd}\to I$ is a functor such that $p|_R = p_I^n|_R$.

As a set, we define: $A_I^{n+1} := A_I^n \coprod B_I^{n+1}$. For $c$ in $ A_I^{n}$, we set $c < (R,p:R^{\lhd}\to I)$ iff $c $ in $R$. Thus we have defined an $(n+1)$-level cofinite poset: $A_I^{n+1}$.
We now define $p_I^{n+1}:A_{I}^{n+1} \to I$ by $p^{n+1}|_{A_{I}^n} = p_I^{n} $ and $p_I^{n+1}(R,p:R^{\lhd}\to T) = p(\infty)$, where $\infty $ is the initial object in $R^{\lhd}$.

We now have an infinite chain of cofinite posets:
$$A_I^{-1} \subset A_I^{0} \subset \cdots \subset A_I^{n} \subset \cdots .$$
We define $A_I := \bigcup A_I^n$.

It is clear that by taking the limit on all the $p_I^n$ we obtain a functor $p_I:A_I \to I$.
\end{define}

Note that $A_I^0=Ob(I)$, and $p_I^0:A_I^{0}=Ob(I) \to I$ is just the identity on $Ob(I)$.

\begin{lem}\label{l:A_is_directed}
Let $I$ be a small directed category. Then $A_I$ is cofinite directed and of infinite height.
\end{lem}
\begin{proof}
It is clear by construction that $A_I$ is  cofinite and of infinite height.
To prove that $A_I$ is directed we need to  show that for every finite section $R \subset A_I$,
there exists an element $c $ in $A_I$, such that $c\geq r$ for every $r$ in $ R$ (see Lemma ~\ref{l:eqiv_directed}).
Indeed let $R \subset A_I$ be a finite section. Since $R$ is finite, there exists some $n $ in $\NN$ such that
$R \subset A_I^{n}$. We can take $c$ to be any element in $B_I^{n+1}$ of the form $(R,p:R^{\lhd}\to T)$. To show that such an element exists, note that since $I$ is directed we can extend the functor $p_I^n|_R:R\to I$ to a functor $p:R^{\lhd}\to I$ (see Lemma ~\ref{l:eqiv_directed}).
\end{proof}

\begin{lem}\label{l:A_onto}
Let $I$ be a small directed category. Then the functor: $p_I:A_I\to I$ is cofinal.
\end{lem}

\begin{proof}
By Definition ~\ref{d:cofinal} we need to show that for every $i$ in $I$, the over category ${p_I}_{/i}$ is nonempty and connected. Let $i$ in $I$.

As noted above, $A_I^0=Ob(I)$ and $p_I|_{A_I^{0}}:Ob(I) \to I$ is the identity on $Ob(I)$. Thus $(i,id_i)$ is an object in ${p_I}_{/i}$.

Let $f_1:{p_I}(a_1)\to i$ and $f_2:{p_I}(a_2)\to i$ be two objects in ${p_I}_{/i}$. Since $A_I$ is directed, there exists $c$ in $A_I$ such that $c\geq a_1, a_2$. Applying ${p_I}$ and composing with $f_1$ and $f_2$ we get two parallel morphisms in $I$: ${p_I}(c)\to i$. Since $I$ is directed, there exists a morphism: $h:i'\to {p_I}(c)$ in $I$ that equalizes these two parallel morphisms.

We now wish to show that there exists $c'$ in $A_I$ such that $c'\geq c$ and such that ${p_I}(c')  =i'$ and the induced map: ${p_I}(c') \to {p_I}(c)$ is exactly $h$.

There exists a unique $n\geq 0$, such that $c$ is in $A_I^n\setminus A_I^{n-1}=B_I^n$.
We can write $c$ as $c = (R,p:R^{\lhd} \to I)$, where $R$ is a finite section in $A_I^{n-1}$.

Note that $R_c:=\{a$ in $A_I^n|c\geq a\}\subseteq A_I^n $ is naturally isomorphic to $R^{\lhd}$.

Define: $c' := (R_c,p':R_c^{\lhd}\to I) $ in $B_I^{n+1}$, where:
$$p'|_{R_c} = p'|_{R^{\lhd}} = p|_{R^{\lhd}},p'(\infty') = i'.$$
The map $p'(\infty') = i'\to {p_I}(c)=p(\infty)=p'(\infty)$ is defined to be $h$ (where $\infty$ and $\infty'$ are the initial objects in $R^{\lhd}$ and $R_c^{\lhd}$ respectively).

To show that $c' $ is in $B_I^{n+1}$, it remains to check that $p'|_{R_c}=p^n|_{R_c}$. But this follows from the fact that $p|_{R}=p^{n-1}|_{R}$, and the (recursive) definition of $p_n$.

Now it is clear that: $c'>c$, ${p_I}(c') = i'$ and the induced map: ${p_I}(c') \to {p_I}(c)$ is exactly $h$.

It follows that we have morphisms in ${p_I}_{/i}$:
$$\xymatrix{
{p_I}(a_1) \ar[dr]_{f_1}  & {p_I}(c') \ar[l] \ar[r] \ar[d] & {p_I}(a_2).  \ar[dl]^{f_2}\\
\empty &  i & \empty }$$
\end{proof}

We thus obtain the following:
\begin{cor}\label{p_onto}
Let $I$ be a small directed category. Then there exists a small cofinite directed set $A_I$ of infinite height and a cofinal functor: $p_I:A_I\to I$.
\end{cor}

Corollary ~\ref{p_onto} is actually a well known result in the theory of pro-categories. In \cite{Isa}, Isaksen gives two references to this proposition: one is \cite{EH} Theorem 2.1.6 and the other is \cite{SGA4-I} Proposition 8.1.6.

We would like to take this opportunity to explain a slight error in the construction of \cite{EH}. We briefly recall the construction of \cite{EH} Theorem 2.1.6.

Let $\cD$ be any category. Call an object $d$ in $\cD$ strongly initial, if it is an initial object, and there are no maps into $d$ except the identity. Define:
$$M(I):=\{\cD\to  I|\cD\: is\: finite,\: and\: has\: a\: strongly\: initial\: object\}.$$

We order the set $M(I)$ by sub-diagram inclusion, so $M(I)$ is clearly cofinite.

Then \cite{EH} claims that because $I$ is directed, $M(I)$ is also directed. Apparently the idea is that given two diagrams: $F_1:\cD_1\to I$ and $F_2:\cD_2\to I$, we can take the disjoint union of $\cD_1$ and $\cD_2$, and add an initial object: $(\cD_1\coprod\cD_2)^{\lhd}$. In order to define a diagram $(\cD_1\coprod\cD_2)^{\lhd}\to I$ extending $F_1$ and $F_2$, it is thus enough to find an object $F(\infty)$ in $I$, and morphisms $F(\infty)\to F_1(\infty_1)$ and $F(\infty)\to F_2(\infty_2)$ in $I$. Since $I$ is directed this can be done. Notice, however, that we have only used the fact that $I$ satisfies one of the axioms of a directed category, namely, that for every pair of objects there is an object that dominates both. If this construction was correct it would mean that for every category $I$ satisfying only the first axiom of a directed category, there exists a directed poset $\cP$ and a cofinal functor $\cP\to I$. This would imply that $I$ is a directed category, by the lemma below.
But there are examples of categories satisfying only the first axiom of a directed category, that are not directed, for example
the category $\bullet \rightrightarrows \bullet$ or the category of hyper covers on a Grothendieck site (see ~\cite{AM}).

The reason why this construction is wrong is that $\cD_1$ and $\cD_2$ may not be disjoint (they may have an object in common), and thus one cannot always consider their disjoint union: $\cD_1\coprod\cD_2$. This may sound like a purely technical problem, since we can ``force" $\cD_1$ and $\cD_2$ to be disjoint, for example by
considering $(\cD_1\times\{0\})\coprod(\cD_2\times\{1\})$. But then $F_1$ and $F_2$ will not be sub-diagrams of $F$, rather there would exist isomorphisms from them to sub-diagrams of $F$. In other words, $M(I)$ will not be a poset.

\begin{lem}
Let $A$ be a directed category, $D$ any category and $F:A\to D$ a cofinal functor. Then $D$ is directed.
\end{lem}

\begin{proof}
By Corollary ~\ref{p_onto} we may assume that $A$ is a directed poset. By Definition ~\ref{d:cofinal}, for every $c$ in $D$, the over category $F_{/c}$ is nonempty and connected.

Let $c$ and $d$ be objects in $D$. The categories $F_{/c}$ and $F_{/d}$ are non-empty, so there exist elements $q$ and $p$ in $A$, and morphisms in $D$ of the form: $F(q)\to d$ and $F(p)\to c.$

$A$ is directed, so there exists $r$ in $A$ such that $r\geq p,q$. Then $F(r)$ is in $D$, and we have morphisms in $D$ of the form:
$$F(r)\to F(q)\to d,F(r)\to F(p)\to c.$$

Let $f,g:c\to d$ be two parallel morphisms in $D$. The category $F_{/c}$ is nonempty, so there exists $p$ in $A$, and a morphism in $D$ of the form: $h:F(p)\to c$. Then $gh$ and $fh$ are in $F_{/d}$, and $F_{/d}$ is connected, so there exists elements in $A$ of the form:
$$p\leq p_1\geq p_2\leq \cdots p_n\geq p,$$
that connect $gh,fh:F(p)\to d$ in the over category $F_{/d}$. The poset $A$ is directed, so there exists $q$ in $A$, such that $q\geq p,p_1,...,p_n$. It follows that we have a commutative diagram in $D$ of the form:
$$\xymatrix{
F(p) \ar[dr]_{gh}  & F(q) \ar[l]_{l_1} \ar[r]^{l_2} \ar[d] & F(p)  \ar[dl]^{fh}\\
\empty &  d & \empty .}$$
But $l_1=l_2=l$, since $A$ is a poset. Define: $t:=hl:F(q)\to c$. Then:
$$ft=fhl=ghl=gt.$$
\end{proof}

We now turn to defining the functor $S:\Pro(\cC)\to \barr{\Pro}(\cC)$ which will be the inverse equivalence to $i$. Recall from Proposition ~\ref{p_full} and Corollary ~\ref{c_faithful} that $i$ is full and faithful.

\begin{define}\label{d:S}
We define two class functions: $s:Ob(\Pro(\cC))\to Ob(\barr{\Pro}(\cC))$ and $\phi:Ob(\Pro(\cC))\to Mor({\Pro}(\cC))$.

Let $X:I\to \cC$ be an object in $\Pro(\cC)$.
In Definition ~\ref{d_onto} we described a construction that produces a small cofinite poset $A_I$ and a functor: $p_I:A_I\to I$. In Lemma ~\ref{l:A_is_directed} we have shown that $A_I$ is a cofinite directed set of infinite height. Thus we can define:
$$s(X):=p_I^*X=X\circ p_I:A_I\to \cC.$$
Clearly $s(X)$ is an object in $\barr{\Pro}(\cC)$.

We now define:
$$\phi(X):=\nu_{p_I,X}:X\to p_I^*X=i(s(X))$$
to be the morphism in $\Pro(\cC)$ defined by $p_I$ (see the discussion following Definition ~\ref{def_pro}). Since the functor $p_I:A_I\to I$ is cofinal, we have by Lemma ~\ref{l:cofinal} that the $\phi(X)$ is an isomorphism.

We can now apply the construction given in Definition ~\ref{d:inverse} and define the functor $S$ to be:
$$S:=G_{i,s,\phi}:\Pro(\cC)\to \barr{\Pro}(\cC).$$
\end{define}

By Lemma ~\ref{l:inverse} $S$ is an inverse equivalence to $i$, so we obtain:

\begin{cor}\label{c_equiv}
The pair of functors:
$$i:\barr{\Pro}(\cC)\rightleftarrows \Pro(\cC):S$$
are inverse equivalences  of categories.
\end{cor}

\subsection{First applications}

In this subsection we present some simple application of our new model for a pro-category.

\begin{lem}
Let $$f= (\alpha_f,\phi_f):X^A  \to Y^B$$ be a 1-morphism in $\widetilde{\Pro}(\cC)$ and let
$\alpha':B \to A$ be a strictly increasing  map  such that $\alpha' \geq \alpha_f$. Then there exists a unique 1-morphism of the form $f' =(\alpha',\psi'):X^A  \to Y^B$  such that $f' \geq f$ (we denote $f' =f \circ \alpha'$).
\end{lem}

\begin{proof}
The uniqueness is clear.
We have a natural transformation in $\cC^B$:
$$p:(\alpha')^*X   \to \alpha_f^*X.$$
By composing it with $\phi_f: \alpha_f^*X \to  Y$ we get  a 1-morphism:
$$ f \circ \alpha' := (\alpha',\phi_f \circ p):X^A  \to Y^B.$$
Note that clearly: $f \circ \alpha' \geq f $.
\end{proof}

\begin{cor}
For every 1-morphism $f=(\alpha_f,\phi_f)$ from $X$ to $Y$ in $\widetilde{\Pro}(\cC)$, we have that the subposet of $Mor_{\widetilde{\Pro}(\cC)}(X,Y)$ spanned by those
$f':X \to Y$ such that $f'\geq f$ is isomorphic to the poset of the strictly increasing functions $\alpha:B \to A$ such that $\alpha\geq\alpha_f$.
\end{cor}

\begin{lem}\label{l:SID}
Let $A,B$ be cofinite directed sets of infinite height. Then the poset of strictly increasing functions from $B$ to $A$ is directed and of infinite height.
\end{lem}

\begin{proof}
Let $\alpha,\beta:B\to A$ be strictly increasing functions. We will construct a strictly increasing function $\gamma:B\to A$ such that $\gamma(b)>\alpha(b),\beta(b)$, for every $b$ in $B$.

We define $\gamma$ recursively.
Let $n\geq 0$. Suppose we have defined a strictly increasing function $\gamma|_{B^{n-1}}:B^{n-1}\to A$ such that $\gamma(b)>\alpha(b),\beta(b)$, for every $b$ in $B^{n-1}.$

Let $b$ be an element in $B^n\setminus B^{n-1}$. The poset $B$ is cofinite so the set $B_b=\{b'\in B|b\geq b'\}$ is finite. Since $A$ is directed and of infinite height, we can find an element $\gamma(b)$ in $A$ such that $\gamma(b)>\gamma(b')$ for every $b'<b$ in $B$, and such that $\gamma(b)> \alpha(b),\beta(b).$ Now combine all the resulting $\gamma(b)$ for different $b$ in $B^n\setminus B^{n-1}$ to obtain the recursive step.
\end{proof}

\begin{prop}\label{p:ffull}
Let $D$ be a category, and consider the natural functor:
$$j=j_D:\barr{\Pro}(\cC^D) \to \barr{\Pro}(\cC)^{D}.$$
Then the following hold:
\begin{enumerate}
\item If $D$ has a finite number of objects then $j_D$ is faithful.
\item If $D$ is finite then $j_D$ is full and  faithful.
\end{enumerate}
\end{prop}
\begin{proof}
Note that an object $X^A$ in $\barr{\Pro}(\cC^D)$ can be considered  as a functor $X:A\times D \to \cC$.
For every object $d$ in $D$ we shall denote by $X_d:A \to \cC$  the restriction of $X$ to $d$ in the second coordinate.

Let $X^A$ and $Y^B$ be objects in $\widetilde{\Pro}(\cC^D)$.

Assume that $D$ has a finite number of objects. We need to show that:
$$j:\Hom_{\barr{\Pro}(\cC^D)}(X^A,Y^B)\to \Hom_{\barr{\Pro}(\cC)^{D}}(j(X^A),j(Y^B))$$
is injective.

Let $f,g:X^A \to Y^B$ be two  maps  in $ \widetilde{\Pro}(\cC^D)$.
Assume $j(f) = j(g):j(X^A) \to j(Y^B)$ in $\barr{\Pro}(\cC)^{D}$. Note that $j(f)$ and $j(g)$ are morphisms in the functor category $\barr{\Pro}(\cC)^{D}$, or in other words natural transformations between the two functors $j(X^A),j(Y^B):D\to \barr{\Pro}(\cC)$. Thus, for every object $d$ in $D$,
we have that  $j(f)(d),j(g)(d):X_d^A \to Y_d^B $ are the same map in $\barr{\Pro(\cC)}$.
We get that for every object $d$ in $D$, there exists some strictly increasing function $\alpha_d : B \to A$ such that $\alpha_d \geq \alpha_{j(f)(d)},\alpha_{j(g)(d)}$ and
$j(f)\circ \alpha_d  = j(g) \circ \alpha_d : \alpha_d^*X \to Y$ in $\cC^B$.

Now choose some strictly increasing function $\alpha':B\to A$ such that
$\alpha' \geq \alpha_d$ for all $d$ in $D$ (see Lemma ~\ref{l:SID}). We get that $f \circ \alpha' = g \circ \alpha'$ in $\widetilde{\Pro}(\cC^D)$, and thus that $[f]=[g]$ in $\barr{\Pro}(\cC^D)$.

Assume now that $D$ is finite. We need to show that:
$$j:\Hom_{\barr{\Pro}(\cC^D)}(X^A,Y^B)\to \Hom_{\barr{\Pro}(\cC)^{D}}(j(X^A),j(Y^B))$$
is surjective.

Let $f:j(X)\to j(Y)$ be a natural transformation. We have, for every object $d$ in $D$,
a map $[\alpha_d,\phi_d] = f(d):X_d^A \to Y_d^B$ in $\barr{\Pro}(\cC)$.

Let $e:d_1 \to d_2$ be a morphism in $D$.
We have morphisms in $\barr{\Pro}(\cC)$:
$$[id_A,\phi_{X_{e}} ] = X_e: X_{{d_1}}^A \to X_{{d_2}}^A,$$
$$[id_B,\phi_{Y_{e}} ] = Y_e: Y_{{d_1}}^B \to Y_{{d_2}}^B.$$
Since $f$ is a natural transformation we have an equality in $\barr{\Pro}(\cC)$:
$$[id_B,\phi_{Y_{e}}]\circ [\alpha_{d_1},\phi_{d_1}]=Y_e \circ f(d_1) =f(d_2) \circ X_e=[\alpha_{d_2},\phi_{d_2}]\circ [id_A,\phi_{X_{e}}].$$
Thus there exists some 1-morphism $(\alpha_e,\phi_e)$ such that:
$$(\alpha_e,\phi_e)\geq(\alpha_{d_2},\phi_{d_2})\circ (id_A,\phi_{X_{e}}),(id_B,\phi_{Y_{e}})\circ (\alpha_{d_1},\phi_{d_1}).$$
Note that $\alpha_e:B \to A$ is a strictly increasing function and we have:
$\alpha_e \geq \alpha_{d_1},\alpha_{d_2}$  and:
$$(\alpha_e,\phi_e)=((\alpha_{d_2},\phi_{d_2})\circ (id_A,\phi_{X_{e}}))\circ \alpha_e,$$
$$(\alpha_e,\phi_e)=((id_B,\phi_{Y_{e}})\circ (\alpha_{d_1},\phi_{d_1}))\circ \alpha_e,$$
$$[\alpha_e,\phi_e] = Y_e \circ f(d_1) = f(d_2) \circ X_e.$$

Now choose some strictly increasing function $\alpha':B \to A$ such that
$\alpha' \geq \alpha_e$ for every morphism $e$ in $D$ (see Lemma ~\ref{l:SID}). In particular, we have that $\alpha' \geq \alpha_d$ for every object $d$ in $D$.

For every object $d$ in $D$ we have:
$$[\alpha',\psi'_d] = [(\alpha_d,\phi_d) \circ \alpha'] =  f(d),$$
and for every morphism $e:d_1 \to d_2$ in $D$ we have:
$$[\alpha',\phi'_e] = [(\alpha_e,\phi_e)\circ \alpha'] =Y_e \circ f(d_1) = f(d_2) \circ X_e.$$

Let $e:d_1 \to d_2$ be a  morphism in $D$. It is not hard to verify that:
$$((\alpha_{d_2},\phi_{d_2})\circ \alpha')\circ (id_A,\phi_{X_e})\geq(\alpha_{d_2},\phi_{d_2})\circ (id_A,\phi_{X_e}).$$
We thus get:
$$(\alpha',\phi'_e)=((\alpha_{d_2},\phi_{d_2})\circ (id_A,\phi_{X_e}))\circ \alpha'=((\alpha_{d_2},\phi_{d_2})\circ \alpha')\circ (id_A,\phi_{X_e}).$$
Similarly we have:
$$(\alpha',\phi'_e)=((id_B,\phi_{Y_{e}})\circ (\alpha_{d_1},\phi_{d_1}))\circ \alpha'=(id_B,\phi_{Y_e})\circ ((\alpha_{d_1},\phi_{d_1})\circ \alpha').$$
We thus have an equality in $\widetilde{\Pro}(\cC)$:
$$(\alpha',\phi_{Y_e}\circ \psi'_{d_1})=(id_B,\phi_{Y_e})\circ (\alpha',\psi'_{d_1}) =(id_B,\phi_{Y_e})\circ ((\alpha_{d_1},\phi_{d_1})\circ \alpha')=
$$
$$=((\alpha_{d_2},\phi_{d_2})\circ \alpha')\circ (id_A,\phi_{X_e})
= (\alpha',\psi'_{d_2})\circ (id_A,\phi_{X_{e}})= (\alpha',\psi'_{d_2}\circ(\phi_{X_e})_{\alpha'}).$$

We shall now construct a 1-morphism  $(\alpha',\psi) = g : X^A \to Y^B$ in $\widetilde{\Pro}(\cC^D)$ such that $j([g]) = f$. $\psi$ can be described as a morphism $\psi:(\alpha')^*X \to Y$ in $\cC^{B\times D}$.
For objects $d$ in $D$ and $b$ in $B$ we take $\psi_{b,d} := (\psi'_d)_b$. To show that $\psi$ is indeed a natural transformation we need to show that for every morphism $e:d_1 \to d_2$  in $D$ and every $b' \geq b$ in $B$, the diagram:
$$
\xymatrix{
X_{\alpha'(b'),d_1}\ar[d] \ar[r] & Y_{b',d_1}\ar[d] \\
X_{\alpha'(b),d_2} \ar[r] & Y_{b,d_2} \\
}
$$
commutes. Indeed, consider the diagram:
$$
\xymatrix{
X_{\alpha'(b'),d_1}\ar[d] \ar[r] & Y_{b',d_1}\ar[d] \\
X_{\alpha'(b),d_1}\ar[d] \ar[r] & Y_{b,d_1}\ar[d] \\
X_{\alpha'(b),d_2} \ar[r] & Y_{b,d_2} \\
}
$$
Now the top square commutes since $\psi'_{d_1}$ is a natural transformation and the bottom one commutes  by the equality:
$$\phi_{Y_e}\circ \psi'_{d_1}=\psi'_{d_2}\circ(\phi_{X_e})_{\alpha'}.$$

It is now easy to verify that indeed  $j([g]) = f$.
\end{proof}

\begin{cor}
Let $D$ be a finite category and let $X:D \to \barr{\Pro}(\cC)$ be a diagram in the image of $j$. Then:
\begin{enumerate}
\item If $\cC$ has finite limits, then the limit of $X$ in $\barr{\Pro}(\cC)$ can be computed levelwise.
\item If $\cC$ has finite colimits, then the colimit of $X$ in $\barr{\Pro}(\cC)$ can be computed levelwise.
\end{enumerate}
\end{cor}

\begin{proof}
We prove (1) and the proof of (2) is identical. Suppose $X:A\to \cC^D$ (we abuse notation and don't write the functor $j$ explicitly). Let $\lim X:A\to \cC$ denote the levelwise limit of $X$, and let $\Delta:\barr{\Pro}(\cC)\to \barr{\Pro}(\cC)^D$
denote the constant (diagonal) functor.

Let $K:B\to \cC$ be an object in $\barr{\Pro}(\cC)$. Clearly $\Delta(K):D \to \barr{\Pro}(\cC)$ is in the image of $F$ so by Proposition ~\ref{p:ffull} we have:
$$\Hom_{\barr{\Pro}(\cC)}(K,\lim X)\cong\Hom_{\barr{\Pro}(\cC^D)}(\Delta(K),X)\cong \Hom_{\barr{\Pro}(\cC)^{D}}(\Delta(K),X).$$
The isomorphism on the left follows from the definition of morphisms in $\barr{\Pro}$.
\end{proof}

We will now prove a result for $\barr{\Pro}$, namely Corollary ~\ref{p:natural}, that is known for the usual $\Pro$ (See \cite{Mey}). The proof for $\barr{\Pro}$ is somewhat simpler.

\begin{define}
We say that a category $D$ is \emph{loopless} if for every object $d$ in $D$ we have $Mor_D(d,d) = \{id_d\}$.
We say that $D$ is \emph{strongly loopless} if it is loopless and in $D$ only equal objects are isomorphic.
\end{define}

The category $D = \Delta^n$ for $n \geq 0$ is an example of a strongly loopless category.

Note that every loopless category has a full strongly loopless subcategory that is equivalent to it (just choose one object from every isomorphism class).

From now until the end of this section we let $D$ be a constant finite strongly loopless category.

\begin{define}\label{d:natural}
Let $F$ be an object in $\barr{\Pro}(\cC)^D$. We will describe a construction that produces a small cofinite directed poset of infinite height $\widetilde{A}_F$, a functor $g(F):\widetilde{A}_F\to \cC^{D}$ and an isomorphism:
$$\phi(F):F\xrightarrow{}j_D(g(F)),$$
in $\barr{\Pro}(\cC)^D$.

We first choose an ordering  $\{d_0,...,d_n\} = Ob(D)$ such that
for $i<j$ we have $\Hom_D(d_i,d_j) = \emptyset$.

For every $0\leq i \leq n$, $F(d_i):A_i \to \cC$ is an object in $\barr{\Pro}(\cC)$.
For every morphism $f$ in $\Hom_D(d_i,d_j)$ we have that:
$$F(f) = [\alpha_{f},\phi_{f}]:F(d_i)\to F(d_j)$$
is a morphism in $\barr{\Pro}(\cC)$ ($\phi_{f}:A_j\to A_i$ is a strictly increasing function).

Consider the cofinite directed set $\prod A := \prod \limits ^n_{i=0}  A_i$. We define $\widetilde{A}_F$ to be its subposet:
$$\widetilde{A}_F := \{ (a_0,...,a_n) \in \prod A |for\: every\: 0\leq i <j \leq n,f\in \Hom_D(d_j,d_i)$$
$$we\: have: a_j \geq \alpha_{f}(a_i) \}.$$

We now define a functor: $X_F:\widetilde{A}_F \times D \to \cC$. On objects $X_F$ is defined by $X_F((a_0,...,a_n),d_i) = F(d_i)_{a_i}$.
Given a map  $ g = ((a'_0,...,a'_n) \geq  (a_0,...,a_n),f:d_j \to d_i)$ in $\widetilde{A}_F \times D$ we define
$$X_F(g): X_F((a'_0,...,a'_n),d_j) \to X_F((a_0,...,a_n),d_i)$$
to be the composition
$$X_F((a'_0,...,a'_n),d_j) = F(d_j)_{a'_j} \to F(d_j)_{a_j} \to$$
$$F(d_j)_{\alpha_{f}(a_i)} \xrightarrow{\phi_{f}} F(d_i)_{a_i} = X_F((a_0,...,a_n),d_i).$$
This is well defined since we have $a'_j \geq a_j \geq \alpha_{f}(a_i)$ (by definition of $\widetilde{A}_F$). We define $g(F)$ to be the functor $g(F):\widetilde{A}_F\to \cC^{D}$ that corresponds to $X_F$.

We will show shortly (see Lemma ~\ref{l:directed2}) that $\widetilde{A}_F$ is a cofinite directed set of infinite height. Thus $g(F)$ is an object in $\barr{\Pro}(\cC^D)$. Clearly the projections $\widetilde{A}_F \hookrightarrow\prod A\to A_i$ induce a morphism $\psi(F):i(F)\xrightarrow{}i(j_D(g(F)))$ in ${\Pro}(\cC)^D$, where $i:\barr{\Pro}(\cC)\to \Pro(\cC)$ is the natural functor
(see the discussion following Definition ~\ref{def_pro}). By Corollary ~\ref{c_equiv}, $\psi(F)$ determines a unique morphism $\phi(F):F\xrightarrow{}j_D(g(F))$ in $\barr{\Pro}(\cC)^D$.
\end{define}

\begin{lem}
For every element $(a_0,...,a_n)$ in $\prod A$ there exists an element $(c_0,...,c_n)$ in $\widetilde{A}_F$ such that
$(c_0,...,c_n) \geq  (a_0,...,a_n)$.
\end{lem}

\begin{proof}
We will construct the element $(c_0,...,c_n)$ recursively.

We first define $c_0:=a_0$. Let $0\leq m\leq n-1$. Suppose we have defined elements $c_0,...,c_m$ in $A_0,...,A_m$ respectively, such that $c_i\geq a_i$ for every $i=0,...,m$ and such that for every $0\leq i <j \leq m$ and every $f$ in $\Hom_D(d_j,d_i)$ we have $c_j \geq \alpha_{f}(c_i)$.

Since $A_{m+1}$ is directed we can find an element $c_{m+1}$ in $A_{m+1}$ such that $c_{m+1}\geq a_{m+1}$ and such that for every $0\leq i <m+1$ and every $f$ in $\Hom_D(d_{m+1},d_i)$ we have $c_{m+1} \geq \alpha_{f}(c_i)$.

Clearly $(c_0,...,c_n)$ satisfies the desired properties.
\end{proof}

\begin{lem}\label{l:directed2}
The poset $\widetilde{A}_F$ is  cofinite, directed and of infinite height.
\end{lem}

\begin{proof}
The poset $\widetilde{A}_F$ is clearly cofinite, being a subposet of $\prod A$. Since $\prod A$ is clearly of infinite height it follows from the previous lemma that $\widetilde{A}_F$ is also of infinite height. To show it is directed, let $(a_0,...,a_n)$ and $(b_0,...,b_n)$ be elements in $\widetilde{A}_F$. Since $\prod A$ is clearly directed, we can find an element $(c'_0,...,c'_n)$ in $\prod A$ such that $(c'_0,...,c'_n) \geq  (a_0,...,a_n),(b_0,...,b_n)$. By the previous lemma, we can find an element $(c_0,...,c_n)$ in $\widetilde{A}_F$ such that $(c_0,...,c_n) \geq  (c'_0,...,c'_n)$.
\end{proof}

From the previous two lemmas we get immediately the following:
\begin{cor}
The inclusion $\widetilde{A}_F \subseteq \prod A$
is cofinal.
\end{cor}

Since it is clear that all the different projections $\prod A \to A_i$ are cofinal, we get that the projections $\widetilde{A}_F \to A_i$ are also cofinal. It follows that the morphism $\psi(F):i(F)\xrightarrow{}i(j_D(g(F)))$ induced by these projections is an isomorphism  (see Lemma ~\ref{l:cofinal}). Thus the corresponding morphism $\phi(F):F\xrightarrow{}j_D(g(F))$ is also an isomorphism.

We now turn to defining the functor $h_D:\barr{\Pro}(\cC)^D\to \barr{\Pro}(\cC^D)$ which will be the inverse equivalence to $j_D$. Recall from Proposition ~\ref{p:ffull} that $j_D$ is full and faithful.

\begin{define}\label{d:h_D}
In Definition ~\ref{d:natural}  we have constructed two class functions: $g:Ob(\barr{\Pro}(\cC)^D)\to Ob(\barr{\Pro}(\cC^D))$ and $\phi:Ob(\barr{\Pro}(\cC)^D)\to Mor(\barr{\Pro}(\cC)^D)$, such that for every object $F$ in $\barr{\Pro}(\cC)^D$ we have that:
$$\phi(F):F\xrightarrow{}j_D(g(F))$$
is an isomorphism in $\barr{\Pro}(\cC)^D$.

We can now apply the construction given in Definition ~\ref{d:inverse} and define the functor $h_D$ to be:
$$h_D:=G_{j_D,g,\phi}:\barr{\Pro}(\cC)^D\to \barr{\Pro}(\cC^D).$$
\end{define}

By Lemma ~\ref{l:inverse}, $h_D$ is an inverse equivalence to $j_D$, so we obtain:

\begin{cor}\label{p:natural}
Let $D$ be a finite strongly loopless category. Then the pair of functors:
$$j_D:\barr{\Pro}(\cC^D) \rightleftarrows  \barr{\Pro}(\cC)^{D}:h_D$$
are inverse equivalences  of categories.
\end{cor}

\begin{rem}
Since taking $\barr{\Pro}$ and functor categories clearly  preserves equivalences, the natural functor $j_D:\barr{\Pro}(\cC^D) \to \barr{\Pro}(\cC)^{D}$ is also an equivalence when $D$ is a finite loopless category.
\end{rem}

\section{Factorizations in pro categories}\label{s_fact}

Recall Proposition ~\ref{p:fact} from the introduction:
\begin{prop}
Let $\mcal{C}$ be a category that has finite limits, and let $N$ and $M$ be classes of morphisms in $\cC$. Then:
\begin{enumerate}
\item $^{\perp}R(Sp^{\cong}(M))=^{\perp}Sp^{\cong}(M)=^{\perp}M.$
\item If $Mor(\mcal{C}) = M\circ N$ then $Mor(\barr{\Pro}(\mcal{C})) = Sp^{\cong}(M)\circ Lw^{\cong}(N)$.
\item If $Mor(\mcal{C}) = M\circ N$ and $N\perp M$ (in particular, if $(N,M)$ is a weak factorization system in $\cC$), then $(Lw^{\cong}(N),R(Sp^{\cong}(M)))$ is a weak factorization system in $\barr{\Pro}(\cC)$.
\end{enumerate}
\end{prop}

The main purpose of this section is to prove the second and third parts of the above proposition. It is done in Propositions ~\ref{p_f} and ~\ref{p_wfs}.

Throughout this section, let $\cC$ be a category that has finite limits and let $N$ and $M$ be classes of morphisms in $\cC$.
We define $Lw^{\cong}(N)$ and $Sp^{\cong}(M)$ in $\barr{\Pro}(\mcal{\cC})$ exactly as in Definition ~\ref{def mor}.

\subsection{Reedy-type factorizations}\label{ss_reedy}

Assume now that  $M\circ N=Mor(\cC)$. The purpose of this subsection is to describe a construction that produces for every cofinite poset $A$ a factorization of the morphisms in $\cC^A$ into a morphism in $Lw(N)$ followed by a morphism $Sp(M)$ (see Definition ~\ref{def natural}). This is done in Definition ~\ref{d:reedy}. We will call this construction the \emph{Reedy construction}. In particular, it will follow that  $Sp(M)\circ Lw(N)=Mor(\cC^A)$.

In constructing this factorization we will use the following:
\begin{lem}\label{l:extend_factorization}
Let $R$ be a finite poset, and let $f:X\to Y$ be a map in $\mcal{C}^{R^{\lhd}}$. Let $X|_R \xrightarrow{g} H \xrightarrow{h} Y|_R$ be a factorization of $f|_R$, such that $g$ is levelwise $N$ and $h$ is special $M$.
Then all the factorizations of   $f$  of the form $X \xrightarrow{g'} H' \xrightarrow{h'} Y$, such that $g'$ is levelwise $N$, $h'$ is special $M$ and $H'|_R=H,g'|_R = g,h'|_R = h$, are in natural 1-1 correspondence with all factorizations of the map   $X(\infty) \to \lim\limits_R H \times_{\lim\limits_R Y} Y(\infty)$ of the form $X(\infty) \xrightarrow{g''} H'(\infty) \xrightarrow{h''} \lim\limits_R H \times_{\lim\limits_R Y} Y(\infty)$, such that $g''$ is in $N$ and $h'' $ in $M$ (in particular there always exists one, since $M\circ N=Mor(\cC)$).
\end{lem}

\begin{proof}
To define a factorization of   $f$  of the form $X \xrightarrow{g'} H' \xrightarrow{h'} Y$ as above, we need to define:
\begin{enumerate}
\item An object: $H'(\infty)$ in $\cC$.
\item Compatible morphisms: $H'(\infty)\to H(r)$, for every $r$ in $R$ (or in other words, a morphism: $H'(\infty)\to\lim\limits_R H$).
\item A factorization $X(\infty) \xrightarrow{g'_{\infty}} H'(\infty) \xrightarrow{h'_{\infty}} Y(\infty)$ of $f_{\infty}:X(\infty)\to Y(\infty)$, s.t:
    \begin{enumerate}
    \item The resulting $g':X\to H',h':H'\to Y$ are natural transformations (we only need to check that the following diagram commutes:
        $$\xymatrix{
        X(\infty) \ar[d] \ar[r] &  H'(\infty) \ar[d] \ar[r] & Y(\infty) \ar[d] \\
        \lim\limits_R X  \ar[r] & \lim\limits_R H \ar[r] & \lim\limits_R Y).
        }$$
    \item $g':X\to H'$ is levelwise $N$ (we only need to check that $g'_{\infty}$ in $N$).
    \item $h':H'\to Y$ is special $M$ (we only need to check the special condition on $\infty\in{R^{\lhd}}$).
    \end{enumerate}
\end{enumerate}
From this the lemma follows easily.
\end{proof}

\begin{define}\label{d:reedy}\textbf{Reedy construction:}
Let $A$ be a cofinite poset and let $f:C\to D$ be a morphism in $\cC^A$. We will describe a construction that produces a factorization of $f$ in $\cC^A$ of the form: $C\xrightarrow{g} H \xrightarrow{h} D$ such that $h$ is in $Sp(M)$ and $g$ is in $Lw(N)$ (see Definition ~\ref{def natural}). We will call it the \emph{Reedy construction}.
We define this factorization of $f$ recursively.

Let $n\geq 0$. Suppose we have defined a factorization  of $f|_{A^{n-1}}$ in $\cC^{A^{n-1}}$ of the form: $C|_{A^{n-1}}\xrightarrow{g|_{A^{n-1}}} H|_{A^{n-1}} \xrightarrow{h|_{A^{n-1}}} D|_{A^{n-1}}$, such that $h|_{A^{n-1}}$ is in $Sp(M)$ and $g|_{A^{n-1}}$ is in $Lw(N)$ (see Definition ~\ref{def_deg}).

Let $c$ be an element in $A^n\setminus A^{n-1}$.

$A^{n-1}_c$ is a finite poset, and $f|_{A_c}:C|_{A_c}\to D|_{A_c}$ is a map in $\mcal{C}^{A_c}$ (see Definition ~\ref{def CDS}). We know that $C|_{A^{n-1}_c}\xrightarrow{g|_{A^{n-1}_c}} H|_{A^{n-1}_c} \xrightarrow{h|_{A^{n-1}_c}} D|_{A^{n-1}_c}$ is a factorization of $f|_{A^{n-1}_c}$, such that $g|_{A^{n-1}_c}$ is levelwise $N$ and $h|_{A^{n-1}_c}$ is special $M$.

Note that $A_c$ is naturally isomorphic to $(A^{n-1}_c)^{\lhd}$.
Thus, by Lemma ~\ref{l:extend_factorization}, every factorization of the map   $C(c) \to \lim\limits_{A^{n-1}_c} H \times_{\lim\limits_{A^{n-1}_c} D} D(c)$
into a map in $N$ followed by a map in $M$ gives rise naturally to a factorization of   $f|_{A_c}$  of the form $C|_{A_c}\xrightarrow{g|_{A_c}} H|_{A_c} \xrightarrow{h|_{A_c}} D_{A_c}$ such that $g|_{A_c}$ is levelwise $N$ and $h|_{A_c}$ is special $M$, extending the recursively given factorization. Choose such a factorization of  $C(c) \to \lim\limits_{A^{n-1}_c} H \times_{\lim\limits_{A^{n-1}_c} D} D(c)$, and combine all the resulting factorizations  of   $f|_{A_c}$  for different $c$ in $A^n\setminus A^{n-1}$ to obtain the recursive step.
\end{define}
\subsection{Factorizations in pro-categories}\label{ss_fact}

The purpose of this subsection is to prove the rest of Proposition ~\ref{p:fact} not proven in Lemma ~\ref{l:SpMo_is_Mo}.

\begin{prop}\label{p_f}
If $Mor(\mcal{C}) = M\circ N$ then $Mor(\barr{\Pro}(\mcal{C})) = Sp^{\cong}(M)\circ Lw^{\cong}(N)$.
\end{prop}

\begin{proof}
Let $f:X\to Y$ be a morphism in $\barr{\Pro}(\mcal{C})$. By Proposition ~\ref{p:natural}, there exists a cofinite directed set $A$ of infinite height and a morphism $f':X'\to Y'$ in $\cC^A$, that is isomorphic to $f$ as a morphism in $\barr{\Pro}(\mcal{C})$. Applying the Reedy construction (see Definition ~\ref{d:reedy}) to $f'$, and composing with the above isomorphisms, we obtain a factorization of $f$ in $\barr{\Pro}(\cC)$ into a morphism in $Lw^{\cong}(N)$ followed by a morphism in $Sp^{\cong}(M)$.
\end{proof}

Our aim now is to prove that if $(N,M)$ is a weak factorization system in $\cC$, then $(Lw^{\cong}(N),R(Sp^{\cong}(M)))$ is a weak factorization system in $\barr{\Pro}(\cC)$. For this we will need the following:

\begin{lem}\label{l_lift}
Assume $Mor(\cC)=M\circ N$. Then:
\begin{enumerate}
\item  $N^{\perp}\subseteq R(M)$.
\item  $^{\perp}M\subseteq R(N)$.
\end{enumerate}
\end{lem}
\begin{proof}
We prove (1) and the proof of (2) is dual.
Let $h:A \to B $ in $N^{\perp}$. We can factor $h$ as:
$$A \xrightarrow{g\in N} C \xrightarrow{f\in M} B.$$
We get the commutative diagram:
$$\xymatrix{
 A \ar[d]_{g\in N}\ar@{=}[r]& A \ar[d]^{h\in N^{\perp}}\\
 C \ar[r]_{f} \ar@{..>}[ru]^{k} & B}$$
where the existence of $k$ is clear. Rearranging, we get:
$$\xymatrix{
 A \ar[dr]^h \ar[r]^g  & C \ar[d]^{f } \ar[r]^k  & A \ar[dl]^h\\
 \empty &  B & \empty,}$$
and we see that $h$ is a retract of $f$ in $M$.
\end{proof}

\begin{lem}\label{l_lift_wfs}
Assume $Mor(\cC)=M\circ N$ and $N \perp M$. Then $(R(N),R(M))$ is a weak factorization system in $\cC$
\end{lem}
\begin{proof}
$^{\perp}M$ and $N^{\perp}$ are clearly closed under retracts, so by Lemma ~\ref{l_lift} we get that: $R(N) \subseteq ^{\perp}M \subseteq R(N)$ and $R(M)\subseteq N^{\perp} \subseteq R(M)$. Now the lemma follows from Lemma ~\ref{c:ret_lift}.
\end{proof}

\begin{lem}\label{l:MN_LWSP}
Assume $N \perp M$. Then $Lw^{\cong}(N) \perp Sp^{\cong}(M) $.
\end{lem}
\begin{proof}
Let $f:X\to Y$ be a map in $Lw^{\cong}(N)$. We want to show that $f$ is in $ {}^\perp Sp^{\cong}(M)$. But ${}^\perp Sp^{\cong}(M) =  {}^\perp M$ by Lemma ~\ref{l:SpMo_is_Mo}, so it is enough to show that there exists a lift in every square in $\barr{\Pro}(\cC)$ of the form:
$$\xymatrix{
 X \ar[d]_{f}\ar[r]& A \ar[d]^{M}\\
 Y \ar[r]     & B.}$$
Without loss of generality, we may assume that $f:X\to Y$ is a natural transformation, which is  a levelwise $N$-map. Thus we have a diagram of the form:
$$\xymatrix{
\{X_t\}_{t\in T} \ar[d]_{f}\ar[r]& A \ar[d]^{{M}}\\
 \{Y_t\}_{t\in T} \ar[r]    &  B.}$$
By the definition of morphisms in $\barr{\Pro}(\cC)$, there exists $t $ in $T$ such that the above square factors as:
$$
\xymatrix{
\{X_t\}_{t\in T} \ar[d]_{f}\ar[r]& X_t \ar[d]^{N}\ar[r]& A \ar[d]^{M}\\
\{Y_t\}_{t\in T}      \ar[r]     & Y_t \ar[r]     &  B.}$$
Since $N\perp M$ we have a lift in the right square of the above diagram, and so a lift in the original square as desired.
\end{proof}

\begin{prop}\label{p_wfs}
If $Mor(\mcal{C}) = M\circ N$ and $N  \perp M$, then $(Lw^{\cong}(N),R(Sp^{\cong}(M)))$ is a weak factorization system in $\barr{\Pro}(\cC)$.
\end{prop}
\begin{proof}
$Mor(\mcal{C}) = M\circ N$ so by Proposition ~\ref{p_f} we have: $Mor(\barr{\Pro}(\mcal{C})) = Sp^{\cong}(M)\circ Lw^{\cong}(N)$.
Thus by Lemmas ~\ref{l_lift_wfs} and ~\ref{l:MN_LWSP} we have: $(R(Lw^{\cong}(N)),R(Sp^{\cong}(M)))$ is a weak factorization system in $\barr{\Pro}(\cC)$.
But by Lemma ~\ref{l:ret_lw}, $R(Lw^{\cong}(N)) = Lw^{\cong}(N)$,
which completes our proof.
\end{proof}

\section{Functorial factorizations in pro categories}\label{s_ff}
The purpose of this section is to prove Theorem ~\ref{t_main}.  It is done in Theorem ~\ref{t_fact} and Corollary ~\ref{c_wfs}.

Throughout this section, let $\cC$ be a category that has finite limits and let $N$ and $M$ be classes of morphisms in $\cC$. We begin with some definitions:

\begin{define}
For any $n\geq 0$ let $\Delta^n$ denote the linear poset: $\{0,...,n\}$, considered as a category with a unique morphism $i\to j$ for any $i\leq j$.
\end{define}

\begin{define}\label{d_ff}
Let $\cD$ be a category.

A functorial factorization in $\cD$  is a section to the composition functor: $\circ:\cD^{\Delta^2}\to \cD^{\Delta^1}$ (which is the pullback to the inclusion: ${\Delta^1}\cong\Delta^{\{0,2\}}\hookrightarrow{\Delta^2}$).

Thus a functorial factorization in $\cD$ consists of a functor:
$\cD^{\Delta^1}\to\cD^{\Delta^2}$ denoted:
$$(X\xrightarrow{f}Y)\mapsto ({X}\xrightarrow{q_f} L_{f}\xrightarrow{p_f}Y)$$
such that:
\begin{enumerate}
\item For any morphism $f$ in $\cD$ we have: $f=p_f\circ q_f$.
\item For any morphism:
$$\xymatrix{{X}\ar[r]^{f}\ar[d]_{l} & {Y}\ar[d]^{k}\\
            {Z} \ar[r]^t   &   {W}}$$
in $\cD^{\Delta^1}$ the corresponding morphism in $\cD^{\Delta^2}$ is of the form
$$\xymatrix{{X}\ar[r]^{q_f}\ar[d]_{l} & L_{f}\ar[r]^{h_f}\ar[d]^{L_{(l,k)}} & {Y}\ar[d]^{k}\\
            {Z} \ar[r]^{q_t} & L_t\ar[r]^{p_t}    &   {W}.}$$
\end{enumerate}

Suppose $\cA$ and $\cB$ are classes of morphisms in $\cD$. The above functorial factorization is said to be into a morphism in $\cA$ followed by a morphism in $\cB$, if for every $f$ in $Mor(\cC)$ we have $q_f$ in $\cA,p_f$ in $\cB$.

\end{define}

\begin{rem}
The definition above of a functorial factorization agrees with the one given in \cite{Rie}. It is slightly stronger than the one given in \cite{Hov} Definition 1.1.1.
\end{rem}

For technical reasons we will also consider the following weaker notion:
\begin{define}\label{d_wff}
Let $\cD$ be a category.

A pseudo-functorial factorization in $\cD$ is a section, up to a natural isomorphism, to the composition functor: $\circ:\cD^{\Delta^2}\to \cD^{\Delta^1}$.

If $\cA$ and $\cB$ are classes of morphisms in $\cD$, we can define a pseudo-functorial factorization into a morphism in $\cA$ followed by a morphism in $\cB$, in the same way as in Definition ~\ref{d_ff}.
\end{define}

\begin{lem}
To any pseudo-functorial factorization in $\cD$ there exists a functorial factorization in $\cD$ isomorphic to it.
\end{lem}
\begin{proof}
Let:
$$(X\xrightarrow{f}Y)\mapsto (\barr{X}\xrightarrow{q_f} L_{f}\xrightarrow{p_f}\barr{Y})$$
be a pseudo-functorial factorization in $\cD$.

There is a natural isomorphism between the identity and the composition of the above factorization with the composition functor. Thus, for any morphism $f$ in $\cD$ we have an isomorphism: $i_f : f\xrightarrow{\cong}p_f\circ q_f$ in $\cD^{\Delta^1}$ denoted:
$$\xymatrix{{X}\ar[r]^{(i_f)_0}\ar[d]_{f} & \barr{X}\ar[d]^{p_f\circ q_f}\\
            {Y} \ar[r]^{(i_f)_1}   &   \barr{Y},}$$
such that for any morphism:
$$\xymatrix{{X}\ar[r]^{f}\ar[d]_{} & {Y}\ar[d]^{}\\
            {Z} \ar[r]^t   &   {W},}$$
in $\cD^{\Delta^1}$ the following diagram commutes:
$$\xymatrix{  & X\ar[dl]^{(i_f)_0}\ar[dd]\ar[rr]^f   &   &   Y\ar[dl]^{(i_f)_1}\ar[dd]\\
          \barr{X}\ar[dd]\ar[rr]^{p_f\circ q_f} &  & \barr{Y}\ar[dd] &           \\
            & Z\ar[dl]^{(i_t)_0}\ar[rr]^t   &   &   W \ar[dl]^{(i_t)_1}\\
          \barr{Z}\ar[rr]^{p_t\circ q_t} &  & \barr{W} & }$$

We define a functorial factorization in $\cD$ by:
$$(X\xrightarrow{f}Y)\mapsto ({X}\xrightarrow{q_f\circ (i_f)_0} L_{f}\xrightarrow{(i_f)_1^{-1}\circ p_f}{Y}).$$
For any morphism $f$ in $\cD$ we have a commutative diagram:
$$\xymatrix{{X}\ar[r]^{q_f\circ (i_f)_0}\ar[d]^{(i_f)_0}_{\cong} & L_{f}\ar[r]^{(i_f)_1^{-1}\circ p_f}\ar[d]_{=} & {Y}\ar[d]^{(i_f)_1}_{\cong}\\
           \barr{X}\ar[r]^{q_f} & L_{f}\ar[r]^{p_f} & \barr{Y},}$$
so the proof is complete.
\end{proof}

\begin{cor}\label{c_wff_to_ff}
Let $\cA$ and $\cB$ be classes of morphisms in $\cD$ that are invariant under isomorphisms. If there exists a pseudo-functorial factorization in $\cD$ into a morphism in $\cA$ followed by a morphism in $\cB$, then there exists a functorial factorization in $\cD$ into a morphism in $\cA$ followed by a morphism in $\cB$.
\end{cor}

We are now ready to prove the first part of Theorem ~\ref{t_main}.

\begin{thm}\label{t_fact}
If $Mor(\mcal{C}) =^{func} M\circ N$ then $Mor(\barr{\Pro}(\mcal{C})) =^{func} Sp^{\cong}(M)\circ Lw^{\cong}(N)$.
\end{thm}

\begin{proof}
Assume that we are given a functorial factorization in $\cC$ into a morphism in $N$ followed by a morphism in $M$. We need to find a functorial factorization in $\barr{\Pro}(\cC)$ into a morphism in $Lw^{\cong}(N)$ followed by a morphism in $Sp^{\cong}(M)$ (see Definition ~\ref{d_ff}).


Since $Lw^{\cong}(N)$ and $Sp^{\cong}(M)$ are clearly invariant under isomorphisms, Corollary ~\ref{c_wff_to_ff} implies that it is enough to find a \emph{pseudo}-functorial factorization in $\barr{\Pro}(\cC)$ into a morphism in $Lw^{\cong}(N)$ followed by a morphism in $Sp^{\cong}(M)$.

Consider now the following commutative diagram of categories:
$$
\xymatrix{
\barr{\Pro}(\cC^{\Delta^2}) \ar[r]^{\sim}_{j_{\Delta^2}} \ar[d]^{\circ_1}& \barr{\Pro}(\cC)^{\Delta^2}\ar[d]^{\circ_2} \\
\barr{\Pro}(\cC^{\Delta^1}) \ar[r]^{\sim}_{j_{\Delta^1}} & \barr{\Pro}(\cC)^{\Delta^1},
}$$
where the $\circ_i$ are the different morphisms induced from composition. The horizontal maps are equivalences by Corollary ~\ref{p:natural}.

We see now that our goal is to construct  a section $s_2$ to $\circ_2$ up to a natural isomorphism.
Note that for this it is enough to find a section $s_1$ to $\circ_1$. Indeed assume we have such an $s_1$. In Definition ~\ref{d:h_D} we have constructed a functor $h_{\Delta^1}$ that is an inverse equivalence to $j_{\Delta^1}$ (see Corollary ~\ref{p:natural}). Thus if we define $s_2 := j_{\Delta^2}\circ s_1\circ h_{\Delta_1}$, we get:
$$ [\circ_2]\circ s_2 = [\circ_2]\circ  j_{\Delta^2}\circ s_1\circ h_{\Delta_1} $$
$$\: \: \: \: \: \: \: \: \: \: \: \: \: \: \: \:= j_{\Delta^1} \circ [\circ_1] \circ s_1\circ h_{\Delta_1} $$
$$\: \: \: \: \: \: \: \: \: \: \: \: \: \: \: \: \: \: \: \: \: \: \:= j_{\Delta^1} \circ h_{\Delta_1} \cong id_{\barr{\Pro}(\cC)^{\Delta^1}}.$$

So we are left with constructing a section to $\circ_1$:

%
%

$$s_1:\barr{\Pro}(\cC^{\Delta^1})\to \barr{\Pro}(\cC^{\Delta^2}).$$

Let $f$ be an object of $\barr{\Pro}(\cC^{\Delta^1})$. Then $f:E^A\to F^A$ is a natural transformation between objects in $\barr{\Pro}(\cC)$. We define the value of our functor on $f$ to be the Reedy factorization of $f$, described in Definition ~\ref{d:reedy}, but where we always use the given functorial factorization in $\cC$:
$$E\xrightarrow{g_f} H_f\xrightarrow{h_f} F.$$
As we have shown, we have: $f=h_f\circ g_f,g_f$ in $ Lw(N),h_f$ in $Sp(M)$.

Let $f$ and $t$ be objects of $\barr{\Pro}(\cC^{\Delta^1})$. Then $f:E^A\to F^A$ and $t:K^B\to G^B$ are natural transformations between objects in $\barr{\Pro}(\cC)$. Let $(\alpha,\Phi)$ be a representative to a morphism $f\to t$ in $\barr{\Pro}(\cC^{\Delta^1})$. Then $\alpha:B\to A$ is a strictly increasing function and $\Phi:\alpha^*f\to t$ is a morphism in $(\cC^{\Delta^1})^B\cong(\cC^B)^{\Delta^1}.$ Thus $\Phi=(\phi,\psi)$ is just a pair of morphisms in $\cC^B$ and we have a commutative diagram in $\cC^B$:
$$\xymatrix{{E\circ\alpha}\ar[r]^{{f_{\alpha}}}\ar[d]_{\phi} & {F\circ\alpha}\ar[d]^{\psi}\\
            K \ar[r]^t    &   G.}$$

Now consider the Reedy factorizations (see Definition ~\ref{d:reedy}) of $f$ and $t$:
$$E\xrightarrow{g_f} H_f\xrightarrow{h_f} F, \: \: \: \: \: \: \: \: K\xrightarrow{g_t} H_t\xrightarrow{h_t} G.$$
We need to construct a representative to a morphism in $\barr{\Pro}(\cC^{\Delta^2})$ between these Reedy factorizations. We take the strictly increasing function $B\to A$ to be just $\alpha$. All we need to construct is a natural transformation: $\chi:H_f\circ\alpha\to H_t$ such that the following diagram in $\cC^B$ commutes:
$$\xymatrix{{E}\circ\alpha\ar[r]^{(g_f)_{\alpha}}\ar[d]_{\phi} & H_{f}\circ\alpha\ar[r]^{(h_f)_{\alpha}}\ar[d]^{\chi} & {F}\circ\alpha\ar[d]^{\psi}\\
            {K} \ar[r]^{g_t} & H_t\ar[r]^{h_t}    &   {G}.}$$
We will define $\chi:H_f\circ\alpha\to H_t$ recursively, and refer to it as the $\chi$-construction.

Let $n\geq 0$. Suppose we have defined a natural transformation: $\chi:(H_f\circ\alpha)|_{B^{n-1}}\to H_t|_{B^{n-1}}$ such that the following diagram in $\cC^{B^{n-1}}$ commutes:
$$\xymatrix{({E}\circ\alpha)|_{B^{n-1}}\ar[r]^{(g_f)_{\alpha}}\ar[d]_{\phi} & (H_{f}\circ\alpha)|_{B^{n-1}}\ar[r]^{(h_f)_{\alpha}}\ar[d]^{\chi} & ({F}\circ\alpha)|_{B^{n-1}}\ar[d]^{\psi}\\
            {K}|_{B^{n-1}} \ar[r]^{g_t} & H_t|_{B^{n-1}}\ar[r]^{h_t}    &   {G}|_{B^{n-1}}}$$
(see Definition ~\ref{def_deg}).

Let $b$ be an element in $B^n\setminus B^{n-1}$.

There exists a unique $m\geq0$ such that $\alpha(b)$ is in $A^m\setminus A^{m-1}$. It is not hard to check, using the induction hypothesis and the assumptions on the datum we began with, that we have an induced commutative diagram:
$$\xymatrix{  & E(\alpha(b))\ar[dl]\ar[dd]\ar[rr]   &   &   K(b)\ar[dl]\ar[dd]\\
          F(\alpha(b))\ar[dd]\ar[rr] &  & G(b)\ar[dd] &           \\
            & \lim_{A_{\alpha(b)}^{m-1}}H_f\ar[dl]^r\ar[rr]   &   &   \lim_{B_{b}^{n-1}}H_t\ar[dl]^s\\
          \lim_{A_{\alpha(b)}^{m-1}}F\ar[rr] &  & \lim_{B_{b}^{n-1}}G}$$
(we remark that one of the reasons for demanding a \emph{strictly} increasing function in the definition of a 1-morphism is that otherwise we would not have the two bottom horizontal morphisms in the above diagram, see Remark ~\ref{r_strictly}).

Thus, there is an induced commutative diagram:
$$\xymatrix{ E(\alpha(b))\ar[d]\ar[r]   &  \lim_{A_{\alpha(b)}^{m-1}}H_f\times_{\lim_{A_{\alpha(b)}^{m-1}}F} F(\alpha(b))\ar[d] \\
 K(b)\ar[r] & \lim_{B_{b}^{n-1}}H_t\times_{\lim_{B_{b}^{n-1}}G} G(b).}$$

We apply the functorial factorizations in $\cC$ to the horizontal arrows in the above diagram and get a commutative diagram:
$$\xymatrix{ E(\alpha(b))\ar[d]\ar[r]   & H_f(\alpha(b))\ar[d]\ar[r]   & \lim_{A_{\alpha(b)}^{m-1}}H_f\times_{\lim_{A_{\alpha(b)}^{m-1}}F} F(\alpha(b))\ar[d] \\
 K(b)\ar[r] & H_t(b)\ar[r] & \lim_{B_{b}^{n-1}}H_t\times_{\lim_{B_{b}^{n-1}}G} G(b).}$$

It is not hard to verify that taking $\chi_b:H_f(\alpha(b))\to H_t(b)$ to be the morphism described in the diagram above completes the recursive definition.

We need to show that the morphism we have constructed in $\barr{\Pro}(\cC^{\Delta^2})$ between the Reedy factorizations does not depend on the choice of representative $(\alpha,\Phi)$ to the morphism $f\to t$ in $\barr{\Pro}(\cC^{\Delta^1})$. So let $(\alpha',\Phi')$ be another 1-morphism from $f$ to $t$.

Thus, $\alpha':B\to A$ is a strictly increasing function, $\Phi'=(\phi',\psi')$ is a pair of morphisms in $\cC^B$ and we have a commutative diagram in $\cC^B$:
$$\xymatrix{{E\circ\alpha'}\ar[r]^{{f_{\alpha'}}}\ar[d]_{\phi'} & {F\circ\alpha'}\ar[d]^{\psi'}\\
            K \ar[r]^t    &   G.}$$

We apply the $\chi$-construction to this new datum and obtain a natural transformation: $\chi':H_f\circ\alpha'\to H_t$.

\begin{lem}\label{l_reedy wd}
Suppose that $(\alpha',\Phi')\geq(\alpha,\Phi)$.
Then for every $b$ in $B$ we have: $\alpha'(b)\geq\alpha(b)$ and the following diagram commutes:
\[
\xymatrix{ & H_f(\alpha(b)) \ar[dr]^{\chi_b} & \\
H_f(\alpha'(b))\ar[ur]\ar[rr]^{\chi'_b} & & H_t(b).}
\]
In other words, we have an inequality of 1-morphisms from $H_f$ to $H_t$:
$$(\alpha',\chi')\geq(\alpha,\chi).$$
\end{lem}

\begin{proof}
$(\alpha',\Phi')\geq(\alpha,\Phi)$ means that for every $b$ in $B$ we have: $\alpha'(b)\geq\alpha(b)$ and the following diagrams commute:
\[
\xymatrix{ & E(\alpha(b)) \ar[dr]^{\phi_b} &  &  &   F(\alpha(b)) \ar[dr]^{\psi_b} &  \\
E(\alpha'(b))\ar[ur]\ar[rr]^{\phi'_b} & & K(b) &  F(\alpha'(b))\ar[ur]\ar[rr]^{\psi'_b} & & G(b).}
\]

We will prove the conclusion inductively.

Let $n\geq 0$. Suppose we have shown that for every $b$ in $B^{n-1}$ the following diagram commutes:
\[
\xymatrix{ & H_f(\alpha(b)) \ar[dr]^{\chi_b} & \\
H_f(\alpha'(b))\ar[ur]\ar[rr]^{\chi'_b} & & H_t(b).}
\]

Let $b$ be an element in $B^n\setminus B^{n-1}$.

There exists a unique $m\geq0$ such that $\alpha(b)$ is in $A^m\setminus A^{m-1}$. As we have shown, we have an induced commutative diagram:
$$\xymatrix{ E(\alpha(b))\ar[d]\ar[r]   &  \lim_{A_{\alpha(b)}^{m-1}}H_f\times_{\lim_{A_{\alpha(b)}^{m-1}}F} F(\alpha(b))\ar[d] \\
 K(b)\ar[r] & \lim_{B_{b}^{n-1}}H_t\times_{\lim_{B_{b}^{n-1}}G} G(b),}$$
and the map $\chi_b:H_f(\alpha(b))\to H_t(b)$ is just the map obtained when we apply the functorial factorizations in $\cC$ to the horizontal arrows in the diagram above:
$$\xymatrix{ E(\alpha(b))\ar[d]\ar[r]   & H_f(\alpha(b))\ar[d]\ar[r]   & \lim_{A_{\alpha(b)}^{m-1}}H_f\times_{\lim_{A_{\alpha(b)}^{m-1}}F} F(\alpha(b))\ar[d] \\
 K(b)\ar[r] & H_t(b)\ar[r] & \lim_{B_{b}^{n-1}}H_t\times_{\lim_{B_{b}^{n-1}}G} G(b).}$$

Similarly, there exists a unique $l\geq0$ such that $\alpha'(b)$ is in $ A^l\setminus A^{l-1}$, and we have an induced commutative diagram:
$$\xymatrix{ E(\alpha'(b))\ar[d]\ar[r]   &  \lim_{A_{\alpha'(b)}^{l-1}}H_f\times_{\lim_{A_{\alpha'(b)}^{l-1}}F} F(\alpha'(b))\ar[d] \\
 K(b)\ar[r] & \lim_{B_{b}^{n-1}}H_t\times_{\lim_{B_{b}^{n-1}}G} G(b).}$$
The map $\chi'_b:H_f(\alpha'(b))\to H_t(b)$ is the map obtained when we apply the functorial factorizations in $\cC$ to the horizontal arrows in the diagram above:
$$\xymatrix{ E(\alpha'(b))\ar[d]\ar[r]   & H_f(\alpha'(b))\ar[d]\ar[r]   & \lim_{A_{\alpha'(b)}^{l-1}}H_f\times_{\lim_{A_{\alpha'(b)}^{l-1}}F} F(\alpha'(b))\ar[d] \\
 K(b)\ar[r] & H_t(b)\ar[r] & \lim_{B_{b}^{n-1}}H_t\times_{\lim_{B_{b}^{n-1}}G} G(b).}$$

Since $\alpha'(b)\geq\alpha(b)$ we clearly have an induced commutative diagram:
$$\xymatrix{  & E(\alpha'(b))\ar[dl]\ar[dd]\ar[rr]   &   &   F(\alpha'(b))\ar[dl]\ar[dd]\\
          E(\alpha(b))\ar[dd]\ar[rr] &  & F(\alpha(b))\ar[dd] &           \\
            & \lim_{A_{\alpha'(b)}^{l-1}}H_f\ar[dl]^r\ar[rr]   &   &   \lim_{A_{\alpha'(b)}^{l-1}}F\ar[dl]^s\\
          \lim_{A_{\alpha(b)}^{m-1}}H_f\ar[rr] &  & \lim_{A_{\alpha(b)}^{m-1}}F}$$

Combining all the above and using the induction hypothesis and the assumptions of the lemma, we get an induced commutative diagram:
$$\xymatrix{  & E(\alpha'(b))\ar[dl]\ar[dd]\ar[r]   &   E(\alpha(b))\ar[dll]\ar[dd]\\
          K(b)\ar[dd] &  &           \\
            &  \lim_{A_{\alpha'(b)}^{l-1}}H_f\times_{\lim_{A_{\alpha'(b)}^{l-1}}F} F(\alpha'(b))\ar[dl]\ar[r]   &   \lim_{A_{\alpha(b)}^{m-1}}H_f\times_{\lim_{A_{\alpha(b)}^{m-1}}F} F(\alpha(b))\ar[dll]\\
          \lim_{B_{b}^{n-1}}H_t\times_{\lim_{B_{b}^{n-1}}G} G(b) &   &  }$$

Applying the functorial factorizations in $\cC$ to the vertical arrows in the diagram above gives us the inductive step.
\end{proof}

We need to show that the morphism we have constructed in $\barr{\Pro}(\cC^{\Delta^2})$ between the Reedy factorizations does not depend on the choice of representative $(\alpha,\Phi)$ to the morphism $f\to t$ in $\barr{\Pro}(\cC^{\Delta^1})$. So let $(\alpha',\Phi')$ be another representative.

Thus, $\alpha':B\to A$ is a strictly increasing function, $\Phi'=(\phi',\psi')$ is a pair of morphisms in $\cC^B$ and we have a commutative diagram in $\cC^B$:
$$\xymatrix{{E\circ\alpha'}\ar[r]^{{f_{\alpha'}}}\ar[d]_{\phi'} & {F\circ\alpha'}\ar[d]^{\psi'}\\
            K \ar[r]^t    &   G.}$$

We apply the $\chi$-construction to this new datum and obtain a natural transformation: $\chi':H_f\circ\alpha'\to H_t$.

The 1-morphisms $(\alpha,\Phi)$ and $(\alpha',\Phi')$ both represent the same morphism $f\to t$ in $\barr{\Pro}(\cC^{\Delta^1})$, so by Corollary ~\ref{c_directed} there exists a  1-morphism $({\alpha''},{\Phi''})$ from $f$ to $t$ such that $({\alpha''},{\Phi''})\geq(\alpha,\Phi),(\alpha',\Phi')$.

Thus, $\alpha'':B\to A$ is a strictly increasing function, $\Phi''=(\phi'',\psi'')$ is a pair of morphisms in $\cC^B$ and we have a commutative diagram in $\cC^B$:
$$\xymatrix{{E\circ\alpha''}\ar[r]^{{f_{\alpha''}}}\ar[d]_{\phi''} & {F\circ\alpha''}\ar[d]^{\psi''}\\
            K \ar[r]^t    &   G.}$$

We apply the $\chi$-construction to this new datum and obtain a natural transformation: $\chi'':H_f\circ\alpha''\to H_t$.

$({\alpha''},{\Phi''})\geq(\alpha,\Phi),(\alpha',\Phi')$ means that for every $b$ in $B$ we have: $\alpha'(b)\geq\alpha(b)$ and the following diagrams commute:

\[
\xymatrix{ & E(\alpha(b)) \ar[dr]^{\phi_b} &  &  &   E(\alpha'(b)) \ar[dr]^{\phi'_b} &  \\
E(\alpha''(b))\ar[ur]\ar[rr]^{\phi''_b} & & K(b) &  E(\alpha''(b))\ar[ur]\ar[rr]^{\phi''_b} & & K(b).}
\]

\[
\xymatrix{ & F(\alpha(b)) \ar[dr]^{\psi_b} &  &  &   F(\alpha'(b)) \ar[dr]^{\psi'_b} &  \\
F(\alpha''(b))\ar[ur]\ar[rr]^{\psi''_b} & & G(b) &  F(\alpha''(b))\ar[ur]\ar[rr]^{\psi''_b} & & G(b).}
\]

Thus, to get the desired result, it remains to show that for every $b$ in $B$ the following diagrams commute:

\[
\xymatrix{ & H_f(\alpha(b)) \ar[dr]^{\chi_b} &  &  &   H_f(\alpha'(b)) \ar[dr]^{\chi'_b} & \\
H_f(\alpha''(b))\ar[ur]\ar[rr]^{\chi''_b} & & H_t(b)  &  H_f(\alpha''(b))\ar[ur]\ar[rr]^{\chi''_b} & & H_t(b).}
\]
But this follows from Lemma ~\ref{l_reedy wd}.

It remains to verify that we have indeed defined a functor.

We first check that the identity goes to the identity.

Let $f$ be an object of $\barr{\Pro}(\cC^{\Delta^1})$. Then $f:E^A\to F^A$ is a natural transformation between objects in $\barr{\Pro}(\cC)$. Clearly $(\alpha,\Phi)=(\alpha,\phi,\psi)=(id_A,id_E,id_F)$ is a representative to the identity morphism $f\to f$ in $\barr{\Pro}(\cC^{\Delta^1})$.

We now need to apply the $\chi$-construction to $(\alpha,\phi,\psi)$. It is not hard to verify that we obtain the identity natural transformation: $\chi=id_{H_f}:H_f\circ\alpha\to H_f$. Thus the result of applying the functor to the identity is the identity.

We now check that there is compatibility with respect to composition.

Let $f,t,r$ be objects of $\barr{\Pro}(\cC^{\Delta^1})$. Then $f:E^A\to F^A,t:K^B\to G^B,r:L^C\to M^C$ are natural transformations between objects in $\barr{\Pro}(\cC)$. Let $(\alpha,\Phi)=(\alpha,\phi,\psi)$ be a representative to a morphism $f\to t$ in $\barr{\Pro}(\cC^{\Delta^1})$ and $(\beta,\Psi)=(\beta,\gamma,\delta)$ be a representative to a morphism $t\to r$ in $\barr{\Pro}(\cC^{\Delta^1})$. Then $$(\alpha\beta,\Psi\Phi_{\beta})=(\alpha\beta,\gamma\phi_{\beta},\delta\psi_{\beta})$$
is a representative to the composition of the above morphisms in $\barr{\Pro}(\cC^{\Delta^1})$.

We now apply the $\chi$-construction to $(\alpha,\phi,\psi)$, and get a natural transformation: $\chi:H_f\circ\alpha\to H_t$, and we apply the $\chi$-construction to $(\beta,\gamma,\delta)$, and get a natural transformation: $\epsilon:H_t\circ\beta\to H_r$.

It is not hard to verify that applying the $\chi$-construction to $(\alpha\beta,\gamma\phi_{\beta},\delta\psi_{\beta})$ yields the natural transformation: $\epsilon\chi_{\beta}:H_f\circ(\alpha\beta)\to H_r$. Thus, applying the $\chi$-construction to the composition $(\beta,\Psi)\circ(\alpha,\Phi)$ yields the composition of the 1-morphisms which are the $\chi$-constructions of $(\beta,\Psi)$ and $(\alpha,\Phi)$, as desired.

\end{proof}

\begin{cor}\label{c_wfs}
If $Mor(\mcal{C}) = ^{func} M\circ N$ and $N\perp M$, then $(Lw^{\cong}(N),R(Sp^{\cong}(M)))$ is a functorial weak factorization system in $\barr{\Pro}(\cC)$.
\end{cor}
\begin{proof}
This follows from Theorem ~\ref{t_fact} and Proposition ~\ref{p_wfs}.
\end{proof}

\begin{rem}
In light of Corollary ~\ref{c_equiv}, Corollary ~\ref{c_wfs} also holds if we replace $\barr{\Pro}(\cC)$ by the usual ${\Pro}(\cC)$.
\end{rem}

Department of Mathematics, University of Muenster, Nordrhein-Westfalen, Germany.
\emph{E-mail address}:
\texttt{ilanbarnea770@gmail.com}

Department of Mathematics, Massachusetts Institute of Technology, Massachusetts, USA.
\emph{E-mail address}:
\texttt{schlank@math.mit.edu}

\end{document}